\documentclass[12pt]{amsart}

\usepackage{amssymb}

\usepackage{amsmath}
\usepackage{amsthm}
\usepackage{color}

\usepackage{txfonts}

\usepackage{enumitem}

\usepackage{graphicx}

\allowdisplaybreaks

\makeatletter
\@namedef{subjclassname@2010}{%
  \textup{2010} Mathematics Subject Classification}
\makeatother

\def\rr{{\mathbb R}}
\def\rn{{{\rr}^n}}

\def\zz{{\mathbb Z}}
\def\nn{{\mathbb N}}

\def\cc{{\mathbb C}}
\def\cs{{\mathcal S}}
\def\cx{{\mathcal X}}

\def\cf{{\mathcal F}}

\def\cl{{\mathcal L}}

\def\cq{{\mathcal Q}}

\def\fz{\infty}
\def\az{\alpha}

\def\loc{{\mathop\mathrm{\,loc\,}}}

\def\lz{\lambda}

\def\gz{{\gamma}}

\def\vz{\varphi}
\def\tz{\theta}

\def\wz{\widetilde}

\def\hs{\hspace{0.3cm}}

\def\ls{\lesssim}

\def\gfz{\genfrac{}{}{0pt}{}}

\def\rn{{{\mathbb R}^n}}
\def\rr{{\mathbb R}}
\def\cc{{\mathbb C}}
\def\zz{{\mathbb Z}}
\def\nn{{\mathbb N}}

\def\hs{\hspace{0.3cm}}

\def\fz{\infty}
\def\az{\alpha}
\def\supp{{\mathop\mathrm{\,supp\,}}}

\def\lz{\lambda}

\def\gz{{\gamma}}
\def\vz{\varphi}
\def\tz{\theta}

\def\wz{\widetilde}

\def\ls{\lesssim}

\def\laz{\langle}
\def\raz{\rangle}

\def\dint{\displaystyle\int}

\def\lf{\left}

\newcommand{\bt}{{B}_{p,q}^{s,\tau}(\rn)}

\newcommand{\sbt}{{b}_{p,q}^{s,\tau}(\rn)}

\newcommand{\cfi}{{\mathcal{F}}^{-1}}

\def\hs{\hspace{0.3cm}}

\usepackage[T1]{fontenc}

\newtheorem{theorem}{Theorem}[section]
\newtheorem{corollary}[theorem]{Corollary}
\newtheorem{lemma}[theorem]{Lemma}
\newtheorem{proposition}[theorem]{Proposition}

\theoremstyle{definition}
\newtheorem{definition}[theorem]{Definition}
\newtheorem{remark}[theorem]{Remark}

\numberwithin{equation}{section}

\begin{document}


\baselineskip=17pt


\title[The  Haar System in Besov-type Spaces]{The  Haar System in Besov-type Spaces}

\author[W. Yuan]{Wen Yuan}
\address{Laboratory of Mathematics and Complex Systems
(Ministry of Education of China)\\School of Mathematical Sciences\\
Beijing Normal University\\
Beijing 100875, People's Republic of China}
\email{wenyuan@bnu.du.cn}

\author[W. Sickel]{Winfried Sickel}
\address{Mathematisches Institut\\Friedrich-Schiller-Universit\"at Jena\\
Jena 07743, Germany}
\email{winfried.sickel@uni-jena.de}

\author[D. Yang]{Dachun Yang\,*}
\address{Laboratory of Mathematics and Complex Systems
(Ministry of Education of China)\\School of Mathematical Sciences\\
Beijing Normal University\\
Beijing 100875, People's Republic of China}
\email{dcyang@bnu.du.cn}

\date{}

\begin{abstract}
Some Besov-type spaces $B^{s,\tau}_{p,q}(\mathbb{R}^n)$ can be characterized
in terms of the behavior of the Fourier--Haar coefficients.
In this article, the authors discuss some necessary restrictions for the parameters
$s$, $\tau$, $p$, $q$ and $n$ of this characterization.
Therefore, the authors  measure the regularity of the characteristic function $\mathcal X$
of the unit cube  in $\mathbb{R}^n$
via the Besov-type spaces $B^{s,\tau}_{p,q}(\mathbb{R}^n)$. Furthermore,
the authors study necessary and sufficient conditions such that the operation
$\langle f, \mathcal{X} \rangle$ generates a continuous linear functional on $B^{s,\tau}_{p,q}(\mathbb{R}^n)$.
\end{abstract}

\subjclass[2010]{Primary 42C15; Secondary 46E35}

\keywords{Besov space, Besov-type space, characteristic function,
orthonormal Haar system, smooth wavelets.
\\
\indent *\, Corresponding author}

\maketitle


\section{Introduction}


Besov-type spaces $\bt$ are generalizations of Besov spaces $B^s_{p,q}(\rn)$.
On the other hand, and more intuitively, they are representing relatives of bmo and so-called $Q$-spaces,
which have been introduced about 30 years ago in complex analysis with applications also in harmonic analysis
and partial differential equations.
As the most transparent special case, let us consider $\bt$ with $p=q$, $\tau\in[0, 1/p)$ and $s\in(0,1)$.
Then a function $f$ belongs to $B^{s,\tau}_{p,p} (\rn)$ if
\[
\sup_{P \in \mathcal{Q},~ |P|\ge 1} \, \frac{1}{|P|^\tau}\, \lf\{\int_P |f(x)|^p dx \right\}^{1/p}< \infty
\]
and
\[
\sup_{P \in \mathcal{Q}} \, \frac{1}{|P|^\tau}\, \lf\{\int_P \int_P \frac{|f(x)-f(y)|^p}{|x-y|^{sp+n}}\, dxdy \right\}^{1/p}< \infty\, ;
\]
see \cite[4.3.3]{ysy}. Here and hereafter, $\mathcal{Q}$ denotes the collection of all  dyadic cubes in $\rn$.
The main philosophy of these Besov-type spaces consists in characterizing the regularity by means of controlling
(weighted) differences of $f$ on cubes.
This makes clear that there must exist a connection to Morrey--Campanato spaces.

To recall the definition  of Besov-type spaces, we let $\vz_0$, $\vz\in\mathcal{S}(\rn)$ be such that
\begin{align}\label{e1.0}
\supp \cf {\vz}_0\subset \{\xi\in\rn:\,|\xi|\le2\}\quad\mathrm{and}\quad
|\cf {\vz}_0(\xi)|\ge C>0\hs\mathrm{if}\hs |\xi|\le \frac53
\end{align}
and
\begin{align}\label{e1.1}
\supp
\cf {\vz}\subset \left\{\xi\in\rn:\,\frac12\le|\xi|\le2\right\} \ \ \mathrm{and}\ \
|\cf {\vz}(\xi)|\ge C>0\ \ \mathrm{if}\ \ \frac35\le|\xi|\le \frac 53,
\end{align}
where $C$ is a positive constant independent of $\vz_0$ and $\vz$.
Observe that there exist positive constants $A$ and $B$ such that
\[
A \le
\cf {\vz_0}(\xi) +
\sum_{j=1}^\infty \cf {\vz} (2^{-j}\xi) \le B \quad \mbox{for any}\quad \xi \in \rn\, .
\]
In what follows, for any $j\in\nn$, we let $\vz_j(\cdot):=2^{jn}\vz(2^j\cdot)$.

For any given $j\in\mathbb{Z}$ and $k\in\mathbb{Z}^n$,
denote by $Q_{j,k}$ the \emph{dyadic cube} $2^{-j}([0,\,1)^n+k)$
and $\ell(Q_{j,k})$
its \emph{side length}. Let
$$\cq:=\lf\{Q_{j,k}:\,j\in\mathbb{Z},
\,k\in\mathbb{Z}^n\right\},\quad \cq^\ast:=\lf\{Q\in\cq:\ell(Q)\le1\right\}$$
and $j_Q:=-\log_2\ell(Q)$ for any $Q\in\cq$.

\begin{definition}\label{d1}
Let $s\in\rr$, $\tau\in[0,\infty)$, $p,$ $q \in(0,\fz]$ and $\vz_0$, $\vz\in\cs(\rn)$ be
as in \eqref{e1.0} and \eqref{e1.1}, respectively.
The \emph{Besov-type space} $\bt$ is defined as the space of all $f\in \mathcal{S}'(\rn)$ such that
$$\|f\|_{\bt}:=
\sup_{P\in\mathcal{Q}}\frac1{|P|^{\tau}}\left\{\sum_{j=\max\{j_P,0\}}^\fz
2^{js q}\left[\int_P
|\vz_j\ast f(x)|^p\,dx\right]^{q/p}\right\}^{1/q}<\fz$$
with the usual modifications made in case $p=\fz$ and/or $q=\fz$.
\end{definition}

\begin{remark}\label{grund}
\begin{enumerate}
\item[(i)] It is known that the  space $\bt$ is a quasi-Banach space (see \cite[Lemma~2.1]{ysy}).

\item[(ii)] It is easy to see that
$B^{s,0}_{p,q}(\rn)$ coincides with the classical Besov space $B^{s}_{p,q}(\rn)$.

\item[(iii)] We have the monotonicity with respect to $s$ and with respect to $q$, namely,
\[
B^{s_0,\tau}_{p,q_0}(\rn)\hookrightarrow B^{s_1,\tau}_{p,q_1}(\rn)\qquad \mbox{if}\quad s_0 >s_1
\quad \mbox{and} \quad q_0,\,q_1\in(0,\infty],
\]
as well as
\[
B^{s,\tau}_{p,q_0} (\rn)\hookrightarrow B^{s,\tau}_{p,q_1}(\rn)  \qquad \mbox{if}\quad q_0 \le q_1\,.
\]

\item[(iv)] Let  $s\in\rr$ and $p\in (0,\fz]$. Then it holds that
\[
B^{s,\tau}_{p,q}(\rn) = B^{s+n(\tau-1/p)}_{\fz,\fz}(\rn)
\]
if either $q\in(0,\fz)$ and $\tau\in(1/p,\fz)$ or
$q=\fz$ and $\tau\in[1/p,\fz)$; see \cite{yy4}.
In case $s+n(\tau-1/p)>0$, the space $B^{s+n(\tau-1/p)}_{\fz,\fz}(\rn)$ is a H\"older-Zygmund space
with a transparent description in terms of differences; see, for instance, \cite[Section 2.5.7]{t83}.

\item[(v)] Since $\bt$ with $\tau>1/p$ is a classical Besov space,   we will mainly consider  the situation
$\tau\in[0,1/p]$ in this article.
\end{enumerate}
\end{remark}

As a generalization of the classical Besov spaces,
the inhomogeneous Besov-type spaces
$B^{s,\tau}_{p,q}(\rn)$,
restricted to the Banach space case, were first introduced by El
Baraka in \cite{el021, el022, el062}. The extension to quasi-Banach
spaces was done in \cite{yy1,yy2}.
Indeed, the homogeneous version of
$B^{s,\tau}_{p,q}(\rn)$ for full parameters was introduced in
 \cite{yy1,yy2} to cover both the Besov
 spaces and the (real-variable) $Q$ spaces   as special cases. Recall that  $Q$ spaces were originally from complex
 analysis  (see \cite{axz,ex,x,x06}) and their real-variable version
 has found  a lot of applications in harmonic analysis  (see, for instance, \cite{ejpx,dx,dx05,x08,kxzz})
and partial differential equations (see, for instance,  \cite{x07,lz,lz10,lz12,ly13,LXY}).
A systematic treatment on the inhomogeneous Besov-type spaces
$\bt$ was later  given in \cite{ysy}. We refer the reader  also to \cite{syy,yy102,yy4,yyz,s011,s011a} for further results on
these spaces. In recent years, the Besov-type spaces and some of their special cases have also found
interesting applications in some partial differential equations such as (fractional) Navier-Stokes equations
(see, for instance, \cite{x07, ly13, LXY, t13b,t14,t15,Le16,Le16-2,Le18}).

Of particular importance for us is the following
embedding into $C_{ub} (\rn)$, the \emph{class
of all complex-valued, uniformly continuous and bounded functions
on $\rn$}.
For its  proof, we refer the reader to \cite[Proposition 2.6(i)]{ysy} and \cite{s011}.

\begin{proposition}\label{grundp}
Let $s\in\rr$, $\tau\in[0,\infty)$ and  $p,\, q \in(0,\fz]$.
\begin{enumerate}
\item[{\rm(i)}] If $s+n(\tau-1/p)>0$, then $B^{s,\tau}_{p,q}(\rn) \hookrightarrow C_{ub} (\rn)$.

\item[{\rm(ii)}]  Let $p\in(0,\fz)$, $q\in(0,\fz]$, $\tau\in(0, 1/p)$ and
$s+ n\tau -\frac np =0$.
Then $
B^{s,\tau}_{p,q}(\rn) \not \subset C_{ub}(\rn) \, .$

\item[{\rm(iii)}] Let $p\in(0,\fz)$ and $q\in(0,\fz]$.
Then $B^{0, 1/p}_{p,q}(\rn)  \not \subset C_{ub}(\rn)\, .$
\end{enumerate}
\end{proposition}

Many classical function spaces can be described via the  Haar system.
Let us mention here at least the works of Haar \cite{h10},  Schauder \cite{s28},  
Marcinkiewicz \cite{m37} and
Ciesielski \cite{c75}, related to $L^p$,  of  Ropela \cite{r76}, Triebel \cite{t73,t78}, 
Oswald \cite{Os79,Os81} and
Kahane and Lemarie \cite{KL}, related to (isotropic) Besov spaces,
of Wojtaszczyk \cite{woj1}, related to Hardy spaces, of Kamont \cite{Ka94,Ka96,Ka97}, treating
anisotropic Besov spaces, and of Seeger and Ullrich \cite{SU1,SU2} and  Garrig\'os
et al. \cite{GSU,GSU18}, investigating  Bessel
potential and Triebel--Lizorkin spaces.
Good sources are also the monographs by Wojtaszczyk \cite[8.3]{woj} and Triebel \cite[Chapter~2]{t10}.

Nowadays it is known that the Haar system is an unconditional Schauder basis in
Besov spaces $B^s_{p,q}(\rn)$ if  $p,q\in(0,\fz)$ and
\[
\max \lf\{ n\lf( \frac 1p -1\right), \frac 1p -1\right\} < s< \min \lf\{1,\frac 1p \right\}\,;
\]
see, for instance, \cite[Theorem~2.21]{t10}. Already in \cite{t78} Triebel has found that,
in cases
\[
 s< \max \lf\{ n\lf( \frac 1p -1\right), \frac 1p -1\right\} \qquad \mbox{or}\qquad
s> \min \lf\{1,\frac 1p \right\},
\]
the Haar system is not an unconditional Schauder basis. For positive and  negative results in  borderline  cases
we refer the reader to Oswald \cite{Os79,Os81,Os18} and Garrig\'os, Seeger and Ullrich \cite{GSU19}.

The pairs $(p,q)$ with $\max\{p,q\}= \infty$ have to be excluded because
the associated Besov spaces are no longer separable.
However, also for those pairs and the associated spaces there exists a characterization in terms of the Fourier--Haar coefficients; see
\cite[Theorem~2.21]{t10}.
This is of particular interest because the Besov-type spaces with $\tau >0$ are always
nonseparable (for all $s$, all $p$ and all $q$).
Finally, we wish to mention that Triebel \cite{t13}
has established a characterization of some classes $\cl^r B^s_{p,q}(\rn)$
(generalizations of Besov spaces also related to Morrey--Campanato spaces, see \cite{t13b,t14})
in terms of the Haar system. For the relations of  $\cl^r B^s_{p,q}(\rn)$  and $\bt$ we refer the reader to \cite{ysy13} and \cite[2.7]{t14}.

The main purpose of this article is to establish some necessary restrictions for the
parameters $s$, $\tau$, $p$, $q$ and $n$ of the characterization of Besov-type spaces via the Fourier--Haar coefficients.
Two obvious restrictions come from  the following properties:
\begin{itemize}
\item the regularity of the characteristic function $\cx$ of the unit cube in Besov-type spaces $B^{s,\tau}_{p,q}(\mathbb{R}^n)$;
\item  the operation
$\langle f, \cx \rangle$ generates a continuous linear functional on $B^{s,\tau}_{p,q}(\mathbb{R}^n)$.
\end{itemize}
We will give answers in terms of restrictions on the
parameters below for which the above two properties hold.

In a continuation \cite{ysy18} of this article, we will discuss sufficient conditions
for characterizing  $B^{s,\tau}_{p,q}(\mathbb{R}^n)$ in terms of Fourier--Haar
coefficients including some applications to pointwise multipliers.

The structure of this article is as follows.
Section \ref{Main} is devoted to a description of our main results
(Theorems \ref{charact}, \ref{limit5}, \ref{limit5b} and \ref{general})
and some comments. We give the necessary and sufficient condition
on the parameters $s$, $p$, $q$ and $\tau$ in Theorem \ref{charact}
so that the characteristic function $\cx$ belongs to $\bt$,
while in Theorems \ref{limit5} and \ref{limit5b}, we
also give the near sharp condition on the parameters so that the operation
$\langle f, \cx\rangle$ generates a continuous linear
functional on $B^{s,\tau}_{p,q}(\mathbb{R}^n)$.

In Section \ref{s3}, we recall some basic
notation and properties of  Besov-type spaces, as well as some tools used to prove the main results.
Our main tool is the wavelet characterization of
$B^{s,\tau}_{p,q}(\mathbb{R}^n)$ in terms of sufficiently smooth Daubechies wavelets.
However, also interpolation ($\pm$-method of Gustavsson and Peetre)
 and characterizations in terms of differences will be used.
In Section \ref{s4}, we give the proof of Theorem \ref{charact}, while Section \ref{s5} is
devoted to the proofs of Theorems \ref{limit5}, \ref{limit5b} and Corollary \ref{klassisch}
as well as Theorem \ref{general}.

Finally, we make some convention on the notation used in this article.
As usual, $\nn$ denotes the \emph{natural numbers}, $\nn_0$
the \emph{natural numbers including $0$},
$\zz$ the \emph{integers} and
$\rr$ the \emph{real numbers}.
We also use $\cc$ to denote the \emph{complex numbers} and $\rn$ the
\emph{$n$-dimensional  Euclidean space}.
All functions are assumed to be complex-valued, namely,
we consider functions
$f:~ \rn \to \cc$.
Let $\mathcal{S}(\rn)$ be the collection of all \emph{Schwartz functions} on $\rn$ equipped
with the well-known topology determined by a countable family of seminorms
and denote by $\mathcal{S}'(\rn)$ its \emph{topological dual}, namely,
the space of all bounded linear functionals on $\mathcal{S}(\rn)$
equipped with the weak-$\ast$ topology.
The symbol $\cf$ refers to  the \emph{Fourier transform},
$\cfi$ to its \emph{inverse transform},
both defined on $\cs'(\rn)$. Recall that, for any $\vz\in \cs(\rn)$ and $\xi\in\rn$,
$$\mathcal{F}\vz(\xi):=(2\pi)^{-\frac n2}\int_\rn e^{-\iota x\xi}\vz(x)\,dx\quad$$
\mbox{and}\quad $$\mathcal{F}^{-1}\vz(\xi):=\mathcal{F}\vz(-\xi),$$ where, for any $x:=(x_1,\ldots,x_n)$ and
$\xi:=(\xi_1,\ldots,\xi_n)\in \rn$, $x\xi:=\sum_{i=1}^nx_i\xi_i$ and $\iota:=\sqrt{-1}$.

All function spaces, which  we consider in this article, are subspaces of $\cs'(\rn)$,
namely, spaces of equivalence classes with respect to
almost everywhere equality.
However, if such an equivalence class  contains a continuous representative,
then usually we work with this representative and call also the equivalence class a continuous function.
By $C^\infty_c(\rn)$ we mean the set of all infinitely differentiable functions on $\rn$ with compact supports.

The \emph{symbol}  $C $ denotes   a positive constant
which depends
only on the fixed parameters $n$, $s$, $\tau$, $p$, $q$ and probably on auxiliary functions,
unless otherwise stated; its value  may vary from line to line.
Sometimes we   use the symbol ``$ \ls $''
instead of ``$ \le $''. The \emph{meaning of $A \ls B$} is
given by that there exists a positive constant $C\in(0,\fz)$ such that
 $A \le C \,B$.
The symbol $A \asymp B$ will be used as an abbreviation of
$A \ls B \ls A$.
Given two quasi-Banach spaces $X$ and $Y$, the operator norm of a linear
operator $T:\, X\to Y$ is denoted by $\|T\|_{X\to Y}$.
Many times we shall use the abbreviation
\begin{align} \label{sigma}
\sigma_p := n \, \max \lf\{0, \frac 1p - 1 \right\},\quad \forall\, p\in(0,\fz].
\end{align}
For any $a\in\rr$, $\lfloor a\rfloor$ denotes the largest integer not greater than $a$.


\section{Main results\label{Main}}


First, we recall the definition of the Haar system. Let $\widetilde{\cx}$ denote the characteristic function of the interval  $[0,1)$.
The generator of the Haar system in dimension $1$ is denoted by $\widetilde{h}$, namely,
\[
\widetilde{h}(t):=\lf\{
\begin{array}{lll}
1 & \qquad & \mbox{when}\quad t\in[0,1/2),
\\
-1 & \quad & \mbox{when}\quad t\in [1/2,1),\\
0  &\quad &\mbox{otherwise}\, .
\end{array}
\right.
\]
The functions, we are interested in, are just tensor products of these two  functions.
For any given $\varepsilon:= (\varepsilon_1, \, \ldots \, , \varepsilon_n)$ with
$\varepsilon_i \in \{0,1\}$, define
\begin{align}
h_\varepsilon (x) := \lf[\prod_{\{i: \, \varepsilon_i=0\}} \widetilde{\cx}(x_i)\right]\,
\lf[\prod_{\{i: \, \varepsilon_i=1\}} \widetilde{h}(x_i)\right]\, ,
\qquad \forall\, x=(x_1, \, \ldots \, , x_n) \in \rn \, .
\end{align}
This results in $2^n$ different functions.
In case $\varepsilon = (0, \, \ldots \, , 0)$  we
\emph{always  write $\cx$ instead of
$h_{(0,\ldots \, , 0)}$}. The other $2^n-1$ functions will be enumerated in an appropriate way
and denoted by $\{h_1, \ldots \, , h_{2^n-1}\}$. These functions
are the generators of the inhomogeneous Haar system in $\rn$.
In what follows, for any $i\in\{1,\ldots,2^n-1\}$,
we let
\begin{align}\label{wavb}
\cx_{j,m}:=2^{jn/2}\cx(2^j\cdot-m),\  \  h_{i,j,m}:=2^{jn/2}h_i(2^j\cdot-m),
\ \ \forall\, j\in\nn_0,\ \forall\, m\in\zz^n.
\end{align}
Then the set
\begin{align*}
H:= \{\cx_{0,m},\,h_{i,j,m}:\ i\in\{1,\ldots,2^n-1\},\ j\in\nn_0,\ m\in\zz^n\}
\end{align*}
forms the well-known orthonormal Haar wavelet system in $\rn$ (we shall call it
just the \emph{Haar system}).

For a function $f\in L^1_\loc(\rn)$ the Haar wavelet expansion is given by
\[
 f=\sum_{m\in\zz^n}\, \langle f , \cx_{0,m}  \rangle \, \cx_{0,m} + \sum_{i=1}^{2^n-1} \sum_{j=0}^\infty
\sum_{m\in\zz^n}\, \langle f , h_{i,j,m} \rangle\, h_{i,j,m}\, .
\]
When trying to characterize a space $\bt$ by the associated Haar wavelet expansions as in \cite[Theorem~2.21]{t10}
there are two obvious necessary conditions:
\begin{itemize}
\item The Haar wavelet coefficients have to be well defined
for any element $f \in \bt$, i.\,e., the mappings $f \mapsto \langle f , \cx_{0,m}  \rangle$
and $f \mapsto \langle f , h_{i,j,m} \rangle$ extend to continuous linear functionals on $\bt$ for any $m,\,i$ and $j$.
\item
The partial sums
\[
S_N f :=
 \sum_{m\in\zz^n}\, \langle f , \cx_{0,m}  \rangle \, \cx_{0,m} + \sum_{i=1}^{2^n-1} \sum_{j=0}^N
\sum_{m\in\zz^n}\, \langle f , h_{i,j,m} \rangle\, h_{i,j,m}
\]
are uniformly bounded in $\bt$, i.\,e.,
\[
\sup_{N \in \nn_0} \sup_{\|f \|_{\bt} \le 1}\, \| S_N f \|_{\bt}<\infty \, .
\]
Of course, this requires $H \subset \bt$.
\end{itemize}

In the framework of the classical Besov spaces,
it is well known that $\cx \in B^s_{p,q} (\rn)$
if and only if
\begin{align}\label{eq-000}
\mbox{either} \quad  s=1/p \quad  \mbox{and} \quad q= \infty
\qquad \mbox{or} \quad s<1/p\quad  \mbox{and} \quad q\in(0,\infty]\, ;
\end{align}
see \cite[Lemma~2.3.1/3]{RS}.
This means that, for fixed $p \in (0,\fz]$,
the smallest Besov space, which the function $\cx$ belongs to, is given by
$B^{1/p}_{p,\infty} (\rn)$.
Now we turn to  the smoothness of $\cx$ and $h_{i,j,m}$ with respect to the scale $\bt$.

\begin{theorem}\label{charact}
Let $s \in \rr$ and $p$, $q\in (0,\fz]$.
\begin{enumerate}
\item[{\rm(i)}] Let $\tau \in(1/p,\fz)$. Then
$\cx \in \bt$ if and only if
$s \le n (1/p - \tau )$.

\item[{\rm(ii)}] Let $\tau \in [0, 1/p]$. Then
$\cx \in \bt$ if and only if
either
\begin{align}\label{eq-001}
s = \frac 1p,\,  \qquad q=\infty \qquad \mbox{and}\qquad  s \le n \lf(\frac 1p - \tau \right)
\end{align}
or
\begin{align}\label{eq-002}
s < \frac 1p,\,  \qquad q\in(0, \infty] \qquad  \mbox{and}
\qquad s \le n \lf(\frac 1p - \tau \right)\, .
\end{align}
\item[{\rm(iii)}] All elements of $H$ have the same smoothness properties with respect to Besov type spaces,
i.\,e., $h_{i,j,m} \in \bt$ if and only if $\cx_{0,m} \in \bt$  if and only if $\cx \in \bt$.
Here $i,\,j,\,m$ are arbitrary (but as in $H$).
\end{enumerate}
\end{theorem}

\begin{remark}\label{rem1}
\begin{enumerate}
\item[(i)] The most interesting thing of Theorem \ref{charact} is the influence of the Morrey parameter $\tau$.
Locally, the smoothness of  elements of $\bt$ grows with $\tau$.
If $\tau $ is large, discontinuous functions can not belong to $\bt$.
What concerns the smoothness of $\cx$, the Morrey parameter comes into play for $\tau \ge \frac{n-1}{np}$.

\item[(ii)] Measuring the regularity of characteristic functions $\cx_\Omega$ of sets $\Omega \subset \rn$ in Besov spaces
has attracted some attention in recent decades (see, for instance, \cite{FR,Gu1,Gu2,RS,s99,s99b,t02,t03,t06}).
These investigations have found some applications in various areas such as  
pointwise multipliers for Besov and Triebel--Lizorkin
spaces (see \cite{Gu1,Gu2,RS,t06}) and the Calder\'on inverse problem (see \cite{FR}).
For a general set $\Omega$, the smoothness of $\cx_\Omega$ depends on the quality of the
boundary. It turns out that the interrelations of smoothness and quality of the
boundary is surprisingly complicated. We shall return to this problem in our forthcoming article \cite{ysy18}.
\end{enumerate}
\end{remark}

Next we consider the mapping
$f\mapsto\langle f, \cx \rangle$ and discuss under which restrictions it extends to a  continuous linear functional on $\bt$.
It seems to be appropriate to distinguish into the cases $p\in[1, \infty]$ and $p\in(0, 1)$.

\begin{theorem}\label{limit5}
Let $s\in\rr$, $p  \in[1,\infty]$, $q\in(0,\fz]$
and $\tau \in [0,\fz)$.
The mapping  $f \mapsto \langle f , \cx  \rangle$  extends from
$B^{s, \tau}_{p,q} (\rn) \cap  L^1 (\rn)$
to a continuous linear functional on $B^{s, \tau}_{p,q} (\rn)$
if and only if either
$$
s= \frac 1p-1\, , \ \ \tau\in\lf[0,\frac{n-1}{np} \right]\ \ \mbox{and}\ \ q\in(0,1]
$$
or
$$s>\frac 1p-1,\ \  \tau\in\lf[0,\frac{n-1}{np}\right]\ \ \mbox{and}\ \ q\in(0,\fz]
$$
or
$$s>\frac np-n\tau-1,\ \ \tau\in  \lf(\frac{n-1}{np},\fz\right)\ \ \mbox{and}\ \ q\in(0,\fz].$$
\end{theorem}

\begin{remark}
\begin{enumerate}
\item[(i)]
In case of classical Besov spaces $B^s_{p,q}(\rn)$ there is a convenient
duality argument to deal with  the existence  $\langle f , \cx  \rangle$ (see \cite{t78}).
It is based on the relation
\[
 \lf(B^s_{p,q}(\rn)\right)' = B^{-s}_{p',q'}(\rn)\, , \qquad 1\le p,\,q<\infty\, , \quad s\in \rr\, .
\]
Let us mention that this formula does not extend to  the spaces $B^{s,\tau}_{p,q}(\rn)$, $\tau >0$.
\item[(ii)] Let again $\tau =0$.
As was mentioned before, Triebel \cite{t78} had shown that $f \mapsto \langle f , \cx  \rangle$  will not extend
to a continuous linear functional on $B^s_{p,q}(\rn)$ if $p  \in[1,\infty]$, $q\in(0,\fz]$ and  $s<\frac 1p -1$.
Kahane and Lemari{\'e}-Rieusset \cite[Part II, Chapt.~6,  Remark 2 on page 349]{KL} have supplemented this dealing with the
special limiting case $s=1/2$ and $p=q=2$.
\end{enumerate}
\end{remark}

The behavior for $p< 1$ is surprisingly different.
However, as in Theorems \ref{charact} and   \ref{limit5},  the value $\tau =  \frac{n-1}{np}$ still
plays a particular role.

\begin{theorem}\label{limit5b}
Let $s\in\rr$, $p  \in (0,1)$ and  $q\in(0,\fz]$.
\begin{enumerate}
\item[{\rm(i)}] Let $\tau\in  (\frac{n-1}{np},\fz)$.
Then  $f \mapsto \langle f , \cx  \rangle$  extends from  $B^{s, \tau}_{p,q} (\rn) \cap  L^1 (\rn)$
to a continuous linear functional on $B^{s, \tau}_{p,q} (\rn)$
if and only if $ s\in(n/p-n\tau-1,\fz)$.

\item[{\rm(ii)}] Let $\tau\in  [0, \frac{n-1}{np}]$.
If either
\[
s =  (1-\tau p)n\lf( \frac 1p -1\right) \qquad \mbox{and}\qquad q\in (0,p]\ \ \ (q\in(0,1]\ \ \mbox{when}\ \ \tau=0)
\]
or
\[
s>(1-\tau p)n\lf( \frac 1p -1\right)  \qquad \mbox{and}\qquad q\in (0,\infty],
\]
then the mapping  $f \mapsto \langle f , \cx  \rangle$  extends from  $B^{s, \tau}_{p,q} (\rn) \cap
 L^1 (\rn)$ to a continuous linear functional on $B^{s, \tau}_{p,q} (\rn)$.
If either
\[
 s = (1-\tau p)n\lf( \frac 1p -1\right) \qquad \mbox{and}\qquad  q \in (1, \infty]
\]
or
\[
  s <  (1-\tau p)n\lf( \frac 1p -1\right)
\]
then the mapping  $f \mapsto \langle f , \cx  \rangle$  does not extend from  $B^{s, \tau}_{p,q} (\rn) \cap L^1 (\rn)$
to a continuous linear functional on $B^{s, \tau}_{p,q} (\rn)$.
\end{enumerate}
\end{theorem}

\begin{remark}
Summarizing, the only case, which has been left open by Theorems \ref{limit5} and \ref{limit5b}, is given by
\[
p\in(0, 1), \quad s =  (1-\tau p)n\lf( \frac 1p -1\right), \quad \tau\in\lf(0,\frac{n-1}{np}\right]  \quad \mbox{and}\quad q\in (p,1].
\]
Clearly, all regions, showing up in the restrictions, are convex, and the dependence on the
parameters is continuous, there exist no jumps.
\end{remark}

For a moment we turn back to the classical situation, namely, we choose $\tau =0$.
Then $B^{s,0}_{p,q} (\rn) = B^{s}_{p,q} (\rn)$ and we obtain, essentially  as a corollary of
Theorems \ref{limit5} and \ref{limit5b}, the following final result.

\begin{corollary}\label{klassisch}
Let $s \in \rr$ and $q  \in (0,\infty]$.
\begin{enumerate}
\item[{\rm(i)}] Let $p\in[1, \infty]$.
Then the mapping  $f \mapsto \langle f , \cx  \rangle$  extends from  $B^{s}_{p,q} (\rn) \cap L^1 (\rn)$
to a continuous linear functional on $B^{s}_{p,q} (\rn)$
if and only if either
$$
s= \frac 1p-1 \qquad \mbox{and}\quad \ q\in(0,1]
$$
or
$$s>\frac 1p-1\qquad  \mbox{and}\quad  \ q\in(0,\fz]\, .$$

\item[{\rm(ii)}]
Let $p  \in (0,1)$.
Then the mapping  $f \mapsto \langle f , \cx  \rangle$  extends from  $B^{s}_{p,q} (\rn) \cap  L^1 (\rn)$
to a continuous linear functional on $B^{s}_{p,q} (\rn)$
if and only if either
\[
s =   n\lf(\frac 1p -1\right)  \qquad \mbox{and}\qquad q\in (0,1]
\]
or
\[
s> n\lf(\frac 1p -1\right) \qquad \mbox{and}\qquad q\in (0,\infty].
\]
\end{enumerate}
\end{corollary}

Finally we turn to the mappings $f \mapsto \langle f , \cx_{0,m}  \rangle$ and
$f \mapsto \langle f , h_{i,j,m}  \rangle$.

\begin{theorem}\label{general}
Theorems \ref{limit5},  \ref{limit5b} and Corollary \ref{klassisch} remain true when replacing
$f \mapsto \langle f , \cx  \rangle$ by either $f \mapsto \langle f , \cx_{0,m}  \rangle$, $m \in \zz^n$
or by $f \mapsto \langle f , h_{i,j,m}  \rangle$, $i \in \{ 1, \ldots, 2^n-1\}, ~ j \in \nn_0, ~m\in \zz^n$.
\end{theorem}

\begin{remark} \begin{itemize}
\item[(a)] Oswald \cite{Os79,Os81,Os18} discussed the properties of the Haar system in limiting cases with $p\in (0,1)$.
He is working on $[0,1]^d$ instead of $\rn$.
In his recent paper \cite{Os18} Oswald proved the following:
If $p\in(\frac{d}{d+1},1)$, $s=d(\frac 1p -1)$ and $q\in (p, \infty)$ then it holds:
\begin{itemize}
 \item[(i)] If $q\in (1,\infty)$, then the coefficient functionals of the Haar expansion can not be extended to bounded linear functionals on
 $B^s_{p,q}([0,1]^d)$.
 \item[(ii)]
 If $q\in (p, 1]$, then the partial sum operators of the Haar expansion are not uniformly bounded on $B^s_{p,q}([0,1]^d)$.
\end{itemize}
Clearly, (i) is the local counterpart of Corollary \ref{klassisch}(ii).
As  mentioned in the Acknowledgement in \cite{Os18},
the counterexamples used to prove (ii) were
communicated to Oswald by Ullrich and they are also published in \cite{GSU19}.

In addition, Oswald was able to show that $H$ restricted to
$[0,1]^d$ is a Schauder basis for $B^s_{p,q}([0,1]^d)$
if $p\in (\frac{d}{d+1}, 1)$, $s=d(\frac 1p -1)$ and $q\in (0, p]$
(this result was also independently obtained by  Garrig\'os, Seeger and Ullrich in \cite{GSU19}).

\item[(b)] Recently, Garrig\'os, Seeger and Ullrich \cite{GSU19}
settled all (!) the borderline cases  of the  Schauder basis properties for the Haar system in $B^s_{p,q}(\rn)$.
\end{itemize}
\end{remark}


\section{Besov-type spaces\label{s3}}

In this section, we recall the  characterizations of
$B^{s,\tau}_{p,q}(\mathbb{R}^n)$ in terms of sufficiently smooth Daubechies wavelets and differences, as well as their interpolation property,
which will be used in our proofs of Theorems \ref{charact}, \ref{limit5} and \ref{limit5b} below.


\subsection{Characterization by wavelets}


Wavelet bases in the Besov and Triebel--Lizorkin spaces are a well
developed concept (see, for instance, Meyer \cite{me},
Wojtasczyk \cite{woj} and Triebel \cite{t06,t08}). Let $\wz{\phi}$ be
an orthonormal scaling function on $\rr$ with compact support
and of sufficiently high regularity. Let $\wz{\psi}$ be one
{\it corresponding orthonormal wavelet}.
Then the tensor product ansatz yields a scaling function $\phi$ and
associated wavelets
$\{\psi_1,\ldots,\psi_{2^n-1}\}$, all defined now on $\rn$ (see, for instance,
\cite[Proposition 5.2]{woj}).
We suppose
\begin{align}\label{4.19}
\phi\in C^{N_1}(\rn)\hs\mathrm{and}\hs
\supp\phi\subset[-N_2,\,N_2]^n
\end{align}
for certain natural numbers $N_1$ and $N_2$. This implies
\begin{align}\label{4.20}
\psi_i\in C^{N_1}(\rn)\hs\mathrm{and}\hs
\supp\psi_i\subset[-N_3,\,N_3]^n,\hs\forall\, i\in\{1,\ldots,2^n-1\}
\end{align}
for some $N_3 \in \nn$.
For any $k\in\zz^n$, $j\in\nn_0$ and $i\in\{1,\ldots,2^n-1\}$, we shall use the
standard abbreviations in this article:
\begin{equation}\label{3.4x}
\phi_{j,k}(x):= 2^{jn/2}\phi(2^jx-k)\hs \quad \mathrm{and} \quad \hs
\psi_{i,j,k}(x):= 2^{jn/2}\psi_i(2^jx-k),\hs \forall\, x\in\rn.
\end{equation}
Furthermore, it is well
known that
\begin{align}\label{moment}
\int_\rn \psi_{i,j,k}(x)\, x^\gz\,dx = 0 \qquad  \mbox{if}
\qquad  |\gz|\le N_1
\end{align}
(see \cite[Proposition 3.1]{woj}) and
\begin{align}\label{4.21}
\{\phi_{0,k}: \ k\in\zz^n\}\: \bigcup \: \{\psi_{i,j,k}:\ k\in\zz^n,\
j\in\nn_0,\ i\in\{1,\ldots,2^n-1\}\}
\end{align}
yields an {\it orthonormal basis} of $L^2(\rn)$; see \cite[Section
3.9]{me} or \cite[Section 3.1]{t06}.
Thus, for any $f\in L^2(\rn)$,
\begin{align}\label{wavelet}
f=\sum_{k\in\zz^n}\, \lambda_k \, \phi_{0,k}+\sum_{i=1}^{2^n-1} \sum_{j=0}^\infty
\sum_{k\in\zz^n}\, \lambda_{i,j,k}\, \psi_{i,j,k}
\end{align}
converges in $L^2(\rn)$,
where  $\lambda_k := \langle f,\,\phi_{0,k}\rangle$ and
$\lambda_{i,j,k}:= \langle f,\,\psi_{i,j,k}\rangle$ with $\langle\cdot,\cdot\rangle$
denoting the inner product of $L^2(\rn)$.
For brevity we put
\begin{align}\label{koeff}
\lambda (f) := \{\lambda_k\}_{k} \cup \{\lambda_{i,j,k}\}_{i,j,k} \, .
\end{align}
By means of such a wavelet system one can discretize the quasi-norm $\|\cdot\|_{\bt}$.
Therefore we need some sequence spaces (see \cite[Definition 2.2]{ysy}).

\begin{definition}\label{dts}
Let $s\in \rr$, $\tau\in[0,\fz)$ and $p,$ $q\in(0,\fz]$.
The sequence space
$\sbt$ is defined to be the space of all
 sequences $t:=\{t_{i,j,m}\}_{i\in\{1\ldots,2^n-1\},j\in\nn_0, m\in\zz^n}\subset \cc$ such that $\|t\|_{\sbt}<\fz$, where
$$\|t\|_{\sbt}:=
\sup_{P\in\mathcal{Q}}\frac1{|P|^{\tau}}\left\{\sum_{j=\max\{j_P,0\}}^\fz
2^{j(s+\frac n2-\frac np)q} \sum_{i=1}^{2^n-1}
\left[\sum_{\{m:\ Q_{j,m}\subset P\}}
|t_{i,j,m}|^p\right]^{\frac qp}\right\}^{\frac 1q} \, .$$
\end{definition}

As a special case of  \cite[Theorem 4.12]{lsuyy} (see also \cite{lsuyy1}), we have the following
wavelet characterization.

\begin{proposition}\label{wav-type2}
Let $s\in\rr$, $\tau\in[0,\fz)$ and $p,\,q\in(0,\fz]$.
Let $N_1\in\nn_0$ satisfy
\begin{align}
 \label{eq-03}
N_1+1&>\max\lf\{n+\frac np-n\tau-s, 2 \sigma_p + 2n+n\tau+1, \right.\\
&\qquad \qquad\qquad\left. n\lf(1+\frac1p+\frac 12\right), n+s, -s+\frac np\right\}.\nonumber
 \end{align}
Let $f\in\cs'(\rn)$. Then
$f\in \bt$ if and only if $f$ can be represented in $\cs'(\rn)$ as in \eqref{wavelet}
such that
\[
\|\lz(f)\|^*_{{b}^{s,\tau}_{p,q}(\rn)}:=
\sup_{P\in\mathcal{Q}}\frac1{|P|^{\tau}}
\left\{\sum_{m:\ Q_{0,m}\subset P}
| \langle f,\,\phi_{0,m}\rangle|^p\right\}^{\frac 1p} +
\|\{\langle f,\,\psi_{i,j,m}\rangle\}_{j,m}\|_{\sbt} <\infty\, .
\]
Moreover, $\|f\|_{\bt}$ is equivalent to
$\|\lz(f)\|^*_{{b}^{s,\tau}_{p,q}(\rn)}$ with the positive equivalence constants independent of $f$.
\end{proposition}

\begin{remark}
\begin{enumerate}
\item[(i)] On the interpretation of $\lambda (f)$, we observe that in general the element  $f$ of $\bt$ is not an element of $L^2 (\rn)$,
it might be a singular distribution. Thus,
$\langle f,\,\phi_{0,m}\rangle$ and $\langle f,\,\psi_{i,j,m}\rangle$ require an interpretation,
which has been done in the proof
of Proposition \ref{wav-type2} (see
\cite[Theorem 4.12]{lsuyy} for all the details).

\item[(ii)] In \cite{lsuyy}, biorthogonal wavelet systems in the sense of
Cohen et al. \cite{codafe92} have been considered.
But here we do not need this generality, orthonormal wavelet systems are sufficient for
our investigations.

\item[(iii)] It is not claimed that the restriction in \eqref{eq-03} is optimal.

\item[(iv)] For the case $s\in(0,\fz)$, we refer the reader also to \cite[Section 4.2]{ysy}.
\end{enumerate}
\end{remark}


\subsection{Characterization by differences}


Historically the characterization by differences (together with some characterizations by approximations)
has been the first description of Besov spaces.
In addition, they look also more transparent than the definition
in terms of convolutions. For that reason the authors of this article have studied those characterizations with a
certain care in \cite{ysy}. To recall one of the results obtained in \cite{ysy},
we first need some notation.

For any $M\in\nn$, function $f: \rn \to \cc$ and $h,\,x\in\rn$, let
\[
\Delta_h^Mf(x):= \sum_{j=0}^M \, (-1)^j\lf(\gfz{M}{j}\right)f(x+(M-j)h)\, ,
\]
where $\big(\gfz{M}{j}\big)$ for any $j\in\{0,\ldots, M\}$ denotes the \emph{binomial coefficients}.
For any $p\in(0,\fz]$, let $L^p(\rn)$ denote the \emph{Lebesgue space} which consists of all measurable functions $f$ such that
$$\|f\|_{L^p(\rn)}:=\lf[\int_\rn |f(x)|^p\,dx\right]^{1/p}<\fz,$$
with the usual modification made when $p=\fz$, and $L_\loc^p(\rn)$ the space of all measurable functions
which belong locally to $L^p(\rn)$.
For any $\tau\in[0, \fz)$, $p\in (0,\fz]$ and $f\in L_\loc^p(\rn)$, let
\begin{align}\label{2.17}
\|f\|_{L^p_\tau(\rn)} := \sup_{P\in\mathcal{Q},\,|P|\ge1}
\frac1{|P|^\tau}
\lf[\dint_P|f(x)|^p\,dx\right]^{1/p},
\end{align}
with the usual modification made when $p=\fz$.
Denote by $L^p_\tau(\rn)$ the set of all
functions $f$ satisfying $\|f\|_{L^p_\tau(\rn)}<\fz$.
Obviously, for any $p\in(0,\fz]$, $L^p_0(\rn)=L^p(\rn)$.
Furthermore, we write
\[
\|f\|^\spadesuit_{\bt}:= \sup_{P\in\mathcal{Q}}
\frac1{|P|^\tau} \, \lf\{\int_0^{2\min\{\ell(P),1\}}t^{-sq}
\sup_{t/2\le|h|<t}\lf[\int_P
|\Delta_h^M f(x)|^p\,dx\right]^{q/p}\,\frac{dt}{t}\right\}^{1/q}\, .
\]

The following difference characterization was proved  in
\cite[Theorems 4.7 and 4.9]{ysy}. Here we focus on the case   $\tau\in[0,1/p]$.

\begin{proposition}\label{t4.7}
Let $q\in(0,\fz]$ and $M\in\nn$.
\begin{enumerate}
\item[{\rm(i)}] Let $p\in [1, \fz]$, $s\in(0,M)$
and $\tau \in[0,1/p].$
Then $f\in\bt$ if and only
if $f\in L^p_\tau(\rn)$ and $\|f\|^\spadesuit_{\bt}<\fz$. Furthermore,
$\|f\|_{L^p_\tau(\rn)}+\|f\|^\spadesuit_{\bt}$ and $\|f\|_{\bt}$ are
equivalent with the positive equivalence constants independent of $f$.

\item[{\rm(ii)}] Let $p\in (0,1)$, $s\in(\sigma_p,M)$
and $\tau \in[0,1/p].$
Then $f\in\bt$ if and only
if $f\in L^p_\tau(\rn)$, $\sup_{\gfz{P\in\cq}{\ell(P)\ge 1}}|P|^{-\tau}\|f\|_{B^{s_0}_{p,\fz}(2P)}<\fz$
with $s_0\in(\sigma_p,s)$ and $\|f\|^\spadesuit_{\bt}<\fz$. Furthermore,
$\|f\|_{L^p_\tau(\rn)}+\sup_{\gfz{P\in\cq}{\ell(P)\ge 1}}|P|^{-\tau}\|f\|_{B^{s_0}_{p,\fz}(2P)}+
\|f\|^\spadesuit_{\bt}$ and $\|f\|_{\bt}$ are
equivalent  with the positive equivalence constants independent of $f$.
\end{enumerate}
\end{proposition}


\subsection{Interpolation of Besov-type spaces}


The interpolation method we shall use is the $\pm$-method introduced by
Gustavsson and Peetre \cite{gp77,g82}. To recall its definition,
we consider a  {couple of quasi-Banach spaces}
(for short, a  {quasi-Banach couple}), $X_0$ and $X_1$,
which are continuously embedded into a larger Hausdorff topological
vector space $Y$. The {space} $X_0+X_1$ is given by
$$X_0+X_1:=\{h\in Y:\ \exists\ h_i\in X_i,\ i\in\{0,1\},\ {\rm such\ that}\ h=h_0+h_1\},$$
equipped with the  {quasi-norm}
\[
\|h\|_{X_0+X_1}:=\inf\lf\{\|h_0\|_{X_0}+\|h_1\|_{X_1}:\ h=h_0+h_1,
\ h_0\in X_1\ {\rm and}\ h_1\in X_1\right\}.
\]

\begin{definition}
Let $(X_0,X_1)$ be a  quasi-Banach couple and $\Theta \in (0,1)$.
An $a\in X_0+X_1$ is said to belong to
$\laz X_0, X_1,\Theta\raz$ if there exists a sequence $\{a_i\}_{i\in\zz}
\subset X_0\cap X_1$ such that $a=\sum_{i\in\zz}\, a_i$ with convergence
in $X_0+X_1$
and, for any finite subset $F\subset \zz$ and any bounded sequence $\{\varepsilon_i\}_{i\in\zz}\subset\cc$,
\begin{align}\label{2.1x}
\lf\|\sum_{i\in F} \varepsilon_i \, 2^{i(j-\Theta)}\, a_i\right\|_{X_j}\le
C \sup_{i\in\zz}|\varepsilon_i|
\end{align}
for some non-negative constant $C$ independent of $F$, $\{\varepsilon_i\}_{i\in\zz}$
and $j\in\{0,1\}$. The {quasi-norm} of
$a\in\laz X_0, X_1, \Theta\raz$  is defined as
\[
\|a\|_{\laz X_0, X_1,\Theta\raz}:=\inf \lf\{C: \, C\ \ \mbox{satisfies}\ \ \eqref{2.1x}\right\}.
\]
\end{definition}

The following property is taken from \cite[Proposition 6.1]{gp77}.

\begin{proposition}\label{gustav}
Let $(A_0,A_1)$ and  $(B_0,B_1)$ be any two quasi-Banach couples and  $\Theta \in (0,1)$.
\begin{enumerate}
\item[{\rm(i)}]  It holds true that $\laz A_0, A_1,\Theta\raz$ is   a quasi-Banach space.

\item[{\rm(ii)}] If $T$ is a linear continuous operator from $A_i$ into $B_i$, $i\in\{0,1\}$, then $T$ maps  $\laz A_0, A_1,\Theta \raz$
continuously into $\laz B_0, B_1,\Theta\raz$. Furthermore,
\[
\| \, T \, \|_{\laz A_0, A_1,\Theta \raz \to \laz B_0, B_1,\Theta \raz} \le \max \lf\{
\| \, T \, \|_{A_0 \to B_0} ,\, \|\, T \, \|_{A_1 \to  B_1}\right\}\, .
\]
\end{enumerate}
\end{proposition}

We also refer the reader  to Nilsson \cite{n85} for  more information on this interpolation method.

The following interpolation  property of Besov-type spaces and the related sequence spaces via the $\pm$ method
was obtained in \cite[Theorem 2.12]{ysy15}.

\begin{theorem}\label{COMI}
Let $\tz\in(0,1)$, $s_i\in\rr$, $\tau_i\in[0,\fz)$
and $p_i$, $q_i\in(0,\fz]$, $i\in\{0,1\},$
be such that $s=(1-\tz)s_0+\tz s_1$, $\tau=(1-\tz)\tau_0+\tz\tau_1$,
\[
\frac1p=\frac{1-\tz}{p_0}+\frac\tz{p_1}\quad \mbox{and}  \quad \frac1q=\frac{1-\tz}{q_0}+\frac\tz{q_1}.
\]
If $\tau_0 \, p_0 = \tau_1\, p_1=\tau p$, then
\begin{equation*}
\lf\laz B_{p_0,q_0}^{s_0,\tau_0}(\rn), B_{p_1,q_1}^{s_1,\tau_1}(\rn),\tz\right\raz=B_{p,q}^{s,\tau}(\rn)
\quad\mbox{and}\quad
\lf\laz b_{p_0,q_0}^{s_0,\tau_0}(\rn), b_{p_1,q_1}^{s_1,\tau_1}(\rn),\tz\right\raz=b_{p,q}^{s,\tau}(\rn).
\end{equation*}
\end{theorem}


\section{Proof of Theorem \ref{charact}\label{s4}}


In this section, by employing the characterizations of Besov-type spaces
via wavelets (see Proposition \ref{wav-type2}) and differences (see Proposition \ref{t4.7})
as well as their interpolation property (see Theorem \ref{COMI}), we give the proof of Theorem \ref{charact}.

First we have to introduce more notation.
For any $Q\in\mathcal{Q}$ and $j\in\zz$, we let
\begin{align}\label{menge1}
J_Q:=\{r\in\zz^n:\ |\supp \phi_{0,r}\cap Q|>0\}
\end{align}
and
\begin{align}\label{menge2}
I_{Q,j}:=\{r\in\zz^n:\ \exists\ i\in\{1,\ldots,2^n-1\}\ \mbox{such that}\
|\supp \psi_{i,j,r}\cap Q|>0\}.
\end{align}

Now we are ready to prove Theorem \ref{charact}.

\begin{proof}[Proof of Theorem \ref{charact}]
We show this theorem by four  steps. Steps 1) - 3) will deal with $\cx$,
in Step 4) the functions $h_{i,j,m}$ are investigated.

{\bf Step 1)} We first consider the case $\tau\in(1/p,\fz)$.
Here it is enough to use the monotonicity of the space
$\bt$ with respect to $s$ and $q$ (see Remark \ref{grund}) and the fact
\[
\cx \in B^0_{\infty,q} (\rn) \qquad \Longleftrightarrow \qquad q=\infty;
\]
see \cite[Lemma ~4.6.3/2]{RS}.
Because of
$B^{s,\tau}_{p,q}(\rn) = B^{s+n(\tau-1/p)}_{\fz,\fz}(\rn)$ (see Remark \ref{grund}(iv)),
we obtain
 $\cx \in B^{s,\tau}_{p,q}(\rn)$ if and only if
$s+n(\tau-1/p) \le 0$.

{\bf Step 2)}
Sufficiency in case $\tau \in [0,1/p]$.
We employ the wavelet characterization of $\bt$
as given in Proposition \ref{wav-type2}. Thus, we have to check
the finiteness of
\begin{align*}
&
\sup_{P\in\mathcal{Q},\ |P|\ge1}\frac1{|P|^\tau}\lf(\sum_{k\in
\mathcal{J}_P}|\langle \cx, \phi_{0,k}\rangle|^p\right)^{\frac 1p}
\\
&\hs+\sup_{P\in\mathcal{Q}}\frac1{|P|^\tau}
\lf\{\sum_{j=\max\{j_P,0\}}^\fz2^{j(s+n/2)q}\sum_{i=1}^{2^n-1}
\lf[\sum_{k\in
\mathcal{I}_{P,j}}2^{-jn}|\langle \cx, \psi_{i,j,k} \rangle|^p\right]^{\frac qp}\right\}^{\frac 1q},
\end{align*}
where $\phi_{0,k}$ and $\psi_{i,j,k} $ are as in \eqref{4.21}.
The first term of the above summation is always finite, hence we may concentrate on the second.
Because of the moment conditions of $\psi_{i,j,k}$ in \eqref{moment},
we conclude that the scalar product
$\langle \cx, \psi_{i,j,k}\rangle=0$ if either $\supp \psi_{i,j,k}
\subset \overline{Q_{0,0}}$ or
$\supp \psi_{i,j,k}  \cap \overline{Q_{0,0}}= \emptyset$.
We define
\[
\Omega_j:= \{r\in \zz^n: \quad
\exists\ i\in\{1,\ldots,2^n-1\}\ \mbox{such that}\,
\supp \psi_{i,j,r} \cap  \partial Q_{0,0} \neq \emptyset\}
\]
{and} $\omega_j := |\Omega_j|$ (the cardinality of $\Omega_j$).
The properties of the wavelet system are guaranteeing
\[
\omega_j \asymp 2^{j(n-1)}\, , \qquad \forall\,j \in \nn_0 \, .
\]
Now we consider two different cases for the size of the dyadic cube $P$.

\emph{Case 1)} Assume that $P\in\cq$ with $|P|\ge 1$.
In this case, we need a few more information about the set ${\mathcal I}_{P,j}$ defined in \eqref{menge2}.
Let $P:= Q_{m,\ell}$ with $m\in\zz\setminus\nn$ and $\ell\in\zz^n$. Then we know that
\[
\mathcal{I}_{P,j} \subset \bigcup_{|\ell-k|\le M} \{r\in\zz^n:\ \exists\ i\in\{1,\ldots,2^n-1\}\ \ \mbox{such that}\ \
|\supp \psi_{i,j,r} \cap Q_{m,k}|>0\}\, ,
\]
where $M$ is a fixed natural number (depending on $N_2$ and $N_3$).
It follows
\[
|\Omega_j \cap \mathcal{I}_{P,j}|\asymp |\Omega_j \cap \mathcal{I}_{Q_{0,0},j}|
\ls 2^{j(n-1)}\, , \qquad \forall\,j \in \nn_0 \,.
\]
In addition we shall use the obvious estimate
\begin{align}\label{eqxx}
|\langle \cx, \psi_{i,j,k} \rangle| &\le 2^{jn/2} \int_{[0,1]^n} | \psi_{i} (2^j x -k)| \, dx
\\
&\le 2^{-jn/2} (\max\{N_2,N_3\})^{n/2}\,
\lf[\int_\rn | \psi_{i} (y)|^ 2 \, dy \right]^{1/2}\ls  2^{-jn/2} \, .\nonumber
\end{align}
Consequently, by \eqref{eq-001}, \eqref{eq-002}, \eqref{eqxx} and the condition on $s$, for those cubes $P$ we have
\begin{align}\label{eq-003}
&\frac1{|P|^\tau}
\lf\{\sum_{j=\max\{j_P,0\}}^\fz2^{j(s+n/2)q}\sum_{i=1}^{2^n-1}
\lf[\sum_{k\in
\mathcal{I}_{P,j}}2^{-jn}|\langle \cx, \psi_{i,j,k} \rangle|^p\right]^{\frac qp}\right\}^{\frac 1q}
\\
&\quad=\frac1{|P|^\tau}
\lf\{\sum_{j=0}^\fz2^{j(s+n/2)q}\sum_{i=1}^{2^n-1}
\lf[\sum_{k\in
\mathcal{I}_{P,j} \cap \Omega_j} \, 2^{-jn}|\langle \cx, \psi_{i,j,k}\rangle|^p\right]^{\frac qp}\right\}^{\frac 1q}
\nonumber
\\
&\quad\ls
\lf\{\sum_{j=0}^\fz2^{j(s+n/2)q}\, 2^{j(n-1)q/p}\, 2^{-jnq/p} \, 2^{-jnq/2}
\right\}^{\frac 1q}
\ls
\lf\{\sum_{j=0}^\fz2^{j(s-1/p)q}\, \right\}^{\frac 1q}<\fz\, .
\nonumber
\end{align}

\emph{Case 2)} Assume now that $P\in\cq$ with $|P|< 1$. We may write
$P:= Q_{m,\ell}$ with $m \in \nn$ and $\ell\in\zz^n$.
Within  cubes of this size,
we have
\[
|\Omega_j \cap \mathcal{I}_{P,j}|\ls |\Omega_j \cap \mathcal{I}_{Q_{m,0},j}|
\ls 2^{(j-m)(n-1)}\, , \qquad \forall\,j \in\{m,m+1,\ldots\}.
\]
From this and \eqref{eq-001} and \eqref{eq-002},  for those cubes $P$ we deduce that
\begin{align*}
 \frac1{|P|^\tau} & \lf\{
\sum_{j=\max\{j_P,0\}}^\fz
2^{j(s+n/2)q}\sum_{i=1}^{2^n-1}
\lf[\sum_{k\in
\mathcal{I}_{P,j}}2^{-jn}|\langle \cx, \psi_{i,j,k}\rangle|^p\right]^{\frac qp}\right\}^{\frac 1q}
\\
&=2^{mn\tau}\,
\lf\{\sum_{j=m}^\fz2^{j(s+n/2)q}\sum_{i=1}^{2^n-1}
\lf[\sum_{k\in
\mathcal{I}_{P,j} \cap \Omega_j} \, 2^{-jn}|\langle \cx, \psi_{i,j,k}\rangle|^p\right]^{\frac qp}\right\}^{\frac 1q}
\\
& \ls  2^{mn\tau}\,
\lf\{\sum_{j=m}^\fz2^{j(s+n/2)q}\, 2^{(j-m)(n-1)q/p}\, 2^{-jnq/p} \, 2^{-jnq/2}
\right\}^{\frac 1q}
\\
& \ls  2^{mn\tau}\, 2^{-m(n-1)/p}
\lf\{\sum_{j=m}^\fz2^{j(s-1/p)q}\, \right\}^{\frac 1q}
\, .
\end{align*}
If either  $s=1/p$ and $q= \infty$ or $s<1/p$ and $q$ arbitrary,  we conclude that
\begin{align}\label{eq-004}
 \frac1{|P|^\tau}
\lf\{
\sum_{j=\max\{j_P,0\}}^\fz2^{j(s+n/2)q}\sum_{i=1}^{2^n-1}
\lf[\sum_{k\in
\mathcal{I}_{P,j}}2^{-jn}|\langle \cx, \psi_{i,j,k}\rangle|^p\right]^{\frac qp}\right\}^{\frac 1q}
\ls 2^{m(s+ n\tau - n/p)}\, ,
\end{align}
which is uniformly bounded in $m$ for $s+ n\tau - n/p\le 0$.
Both estimates together, namely, \eqref{eq-003} and \eqref{eq-004},
prove the sufficiency in cases (i) and (ii).

{\bf Step 3)} Necessity in case $\tau \in [0,1/p]$.
It seems to be difficult to apply the wavelet decomposition
because we do not know how many scalar products
satisfy the inequality
\[
 |\langle \cx, \psi_{i,j,k}\rangle| \ge c \, 2^{-jn/2}
\]
with some positive constant $c$ independent of $j$ and $k$.
For that reason we switch to differences (see Proposition \ref{t4.7}).
Since in case $p=\infty$ the claim is already known [see \eqref{eq-000}],
we may assume $p< \infty$.
In addition we mention that
the necessity of $s + n (\tau -1/p)\le 0$ follows from
Proposition \ref{grundp}(i).
It remains to deal with the relation between $s$ and $1/p$. By the embedding in
Remark \ref{grund}(iii), we only need to show
that $\cx$ is not in $B^{1/p,\tau}_{p,q}(\rn)$ with any given $\tau\in[0,1/p]$ and $q\in(0,\fz)$.

{\bf Substep 3.1)} First we assume that $p>(n-1)/n$, that is,
$\sigma_p < 1/p$.
In this situation we can employ  Proposition \ref{t4.7}.
By using the abbreviations from there and by choosing
$P=Q_{0,0}$, we find that
\begin{align*}
\sup_{t/2\le|h|<t}\lf\{\int_{Q_{0,0}}
|\Delta_h^M \cx (x)|^p\,dx\right\}^{\frac1p}\ge
\sup_{\gfz{t/2\le -h_1 <t}{h_i=0,\ i\in\{2,\ldots,n\}}} \lf\{\int_{\gfz{x_1\in[0, t/2)} {x_i\in[0,1],\ i\in\{2,\ldots,n\}}}
|\Delta_h^M \cx (x)|^p\,dx\right\}^{\frac1p}
\ge  (t/2)^{\frac1p}
\end{align*}
for any $t\in(0,1)$. This immediately implies
that $\|\, \cx \, \|^\spadesuit_{B^{1/p,\tau}_{p,q}(\rn)}= \infty$
for any $q\in(0, \infty)$.

{\bf Substep 3.2)} Now we consider the case $p\le(n-1)/n$, that is,
$\sigma_p \ge1/p$.
In addition, we may assume $\tau\in(0,1/p]$, due to the known result when $\tau=0$ [see \eqref{eq-000}].
 We prove $\cx\notin B^{1/p,\tau}_{p,q}(\rn)$ in this case by contradiction.

Assume that $\cx\in B^{1/p,\tau}_{p,q}(\rn)$ with some $p\le(n-1)/n$ and $q\in(0,\infty)$. We argue by
employing the $\pm$ interpolation method of Gustavsson--Peetre, in particular we shall use
\[
\left\langle B^{1/p, \, \tau}_{p,q} (\rn), B^{1/p_1, \, \tau_1}_{p_1,\infty} (\rn), \theta  \right\rangle
= B^{1/p_0, \, \tau_0}_{p_0,q_0}(\rn)\, ,
\]
with $\theta\in(0, 1)$, $\tau_0= (1-\theta)\, \tau + \theta \, \tau_1$,
\[
\frac 1{p_0} = \frac{1-\theta}{p} + \frac{\theta}{p_1}\, , \qquad
\frac 1{q_0} = \frac{1-\theta}{q} \qquad
\mbox{and}\qquad \frac{\tau}{p_1} = \frac{\tau_1}{p}\,;
\]
see Theorem \ref{COMI}.
We choose  $p_1> (n-1)/n$  and define
\[
\tau_1 := \frac{\tau \, p}{p_1} \, .
\]
Then, by Step 2), we know that  $\cx$ belongs to
$B^{1/p_1, \, \tau_1}_{p_1,\infty} (\rn)$, which, together with the assumption
$\cx\in B^{1/p,\tau}_{p,q}(\rn)$ and the above interpolation formula,
implies that
$\cx \in B^{1/p_0, \, \tau_0}_{p_0,q_0} (\rn)$ for some $q_0\in(0,\infty)$.
Taking $p_1$ as large  and $\theta$ as close to $0$ as we want, we arrive at a situation where also $p_0> (n-1)/n$.
But because of $q_0< \infty$ this is in contradiction to Substep 3.1).

Combining Substeps 3.1) and 3.2), we then know that
$\cx\notin B^{1/p,\tau}_{p,q}(\rn)$ whenever $q\in(0,\fz)$. This finishes the proof of  Theorem \ref{charact} restricted to $\cx$.

{\bf Step 4)} By the translation invariance of the Besov type spaces we have $\cx \in \bt$ if and only if $\cx_{0,m} \in \bt$.
Now we turn to the functions $h_{i,j,m}$.
The if-part follows completely analogous to  Step 1) - Step 3) by concentrating on $h_{i,0,0}$.
Next  we deal with the mapping $f \mapsto f(2\, \cdot\, )$.
Essentially as a consequence of the flexibility in choosing the system $\{\varphi_j\}_{j\in\nn_0}$ in the definition of the spaces $\bt$
we conclude that this mapping is bounded on $\bt$ for all admissible
parameters.
By taking into account the translation invariance of the Besov type spaces this   yields that the functions $h_{i,j,m}$ belong to
$\bt$ whenever $\cx $ is in $\bt$.

Concerning the only if-part  we argue as follows.
By the translation and the rotation invariance of $\bt$ it will be sufficient to deal with
$h_{1,j,0} $. The support of this function is given by
$[0,2^{-j}]^n$. There exists at least one dyadic subcube $Q_{j+1,k}$  such that
$h_{1,j,0}= 1 $  on this cube.
Without loss of generality we may assume $Q_{j+1,k}=Q_{j+1,0}$.
Let $\varrho $ be a compactly supported smooth
function such that $$\supp \varrho \subset \overline{Q_{j+1,0}} \cup \{x: ~ x_1\le 0 , \ldots, \, x_n \le 0\}$$
and $\varrho =1$ on $[2^{-j-3} , 2^{-j-2}]\times [-2^{-j-3}, 2^{-j-3}]^{n-1}$.
Those smooth functions $\varrho$ are pointwise multipliers for $\bt$ (see \cite[6.1.1]{ysy}).
Hence, if we assume $h_{1,j,0} \in \bt$ then $\varrho \cdot h_{1,j,0}\in \bt $ follows.
But locally, more exactly  on 
$$[2^{-j-3} , 2^{-j-2}]\times [-2^{-j-3}, 2^{-j-3}]^{n-1},$$
the product $\varrho \cdot h_{1,j,0}$ behaves like a characteristic function.
Following the arguments in Step 1) - Step 3) one can show that this can only be true
if $\cx $ itself belongs to $\bt$. This finishes the proof of Theorem \ref{charact}
\end{proof}


\section{Proofs of Theorems \ref{limit5}, \ref{limit5b} and \ref{general} as well as Corollary \ref{klassisch}\label{s5}}


This section is devoted to the proofs of Theorems \ref{limit5}, \ref{limit5b} and \ref{general}.
To be precise, in Section \ref{s5.1}, we give the proof of the sufficient conditions
for the existence of $ \langle f , \cx  \rangle$, while the proof of the necessary
conditions for the existence of $ \langle f , \cx  \rangle$ is presented in Section \ref{s5.2}.
Section \ref{s5.3} is devoted to the proofs of Theorems \ref{limit5} and \ref{limit5b}.
Finally the proofs of Corollary \ref{klassisch} and Theorem \ref{general}
are presented, respectively, in Sections \ref{s5.4} and \ref{s5.5}.

To prove these theorems, we shall first
discuss a reasonable way to define $\langle f, \cx \rangle $ for any $f \in B^{s, \tau}_{p,q}(\rn)$.
Fix $s,$ $p,$ $q$ and $\tau$, and let the wavelet system
$$\lf\{\phi_{0,k}, \psi_{i,j,k}:\  k \in \zz^n,\ j \in
\nn_0,\ i\in\{1, \ldots, 2^n-1\}\right\}$$
be admissible for $B^{s,\tau}_{p,q} (\rn)$ in the sense
of Proposition \ref{wav-type2}. Then the wavelet decomposition of $f$ is given by
\begin{align*}
f  = \sum_{k\in\zz^n}\, \langle f, \phi_{0,k} \rangle \, \phi_{0,k}+\sum_{i=1}^{2^n-1} \sum_{j\in\nn_0}
\sum_{k\in\zz^n}\, \langle f,\psi_{i,j,k} \rangle\, \psi_{i,j,k} =  \lim_{N \to \infty} S_N f
\end{align*}
with convergence in $\cs'(\rn)$, where
\[
 S_N f:=  \sum_{k\in\zz^n}\, \langle f, \phi_{0,k} \rangle \, \phi_{0,k}+\sum_{i=1}^{2^n-1} \sum_{j=0}^N
\sum_{k\in\zz^n}\, \langle f,\psi_{i,j,k} \rangle\, \psi_{i,j,k} \, .
\]
Observe that the above  both summations $\sum_{k\in\zz^n}$ are locally finite and hence $S_N f \in C^{N_1}(\rn)$
[due to \eqref{4.19} and \eqref{4.20}].
Then we define
\begin{align}\label{limit1}
 \langle f , \cx  \rangle := \lim_{N \to \infty} \langle S_N f , \cx  \rangle
= \lim_{N \to \infty} \int_{[0,1]^n} S_N f (x)\, dx
\end{align}
whenever this limit exists.


\subsection{Sufficient conditions for the existence of $ \langle f , \cx  \rangle$\label{s5.1}}


Now we turn to the sufficient condition for this existence which at the same time guarantees the independence of
$\langle f , \cx  \rangle$ from the chosen wavelet system.

\begin{theorem}\label{limit2}
Let $q\in (0,\infty]$ and $\tau \in [0,\infty)$.
\begin{enumerate}
\item[{\rm(i)}] Let $ p\in[1, \infty]$.
If $s= \frac 1p - 1$, then $f \mapsto \langle f , \cx  \rangle$  extends to a continuous linear functional on $B^{s, \tau}_{p,1} (\rn)$
which coincides on $B^{s, \tau}_{p,1} (\rn) \cap L^r(\rn)$, $r\in[1, \infty]$, with $\int_\rn f (x)  \cx (x) \, dx$.

\item[{\rm(ii)}] Let $p\in(0, 1)$ and $\tau \in [0, \frac{n-1}{np}]$.
If $s= (1-p \, \tau )n(\frac 1p - 1)$ and $q\in (0,p]$,
then $f \mapsto \langle f , \cx  \rangle$  extends to a continuous linear functional on $B^{s, \tau}_{p,q} (\rn)$
which coincides on $B^{s, \tau}_{p,q} (\rn) \cap L^r(\rn)$, $r\in[1, \infty]$, with $\int_\rn f (x)  \cx (x) \, dx$.

\item[{\rm(iii)}] Let $p\in(0,1)$ and $\tau =0$.
If $s= n(\frac 1p - 1)$ and $q\in (0,1]$,
then $f \mapsto \langle f , \cx  \rangle$  extends to a continuous linear functional on $B^{s, 0}_{p,1} (\rn)$
which coincides on $B^{s,0}_{p,q} (\rn) \cap L^r(\rn)$, $r\in[1, \infty]$, with $\int_\rn f(x)  \cx (x) \, dx$.
\end{enumerate}
\end{theorem}

\begin{proof}
\noindent{\bf Step 1)} Proof of (i).
Let $ p \in[1,\infty]$.
With $P := [0,1]^n$, we find
\[
\langle S_Nf , \cx  \rangle =
\sum_{k\in J_P}\, \langle f,\phi_{0,k} \rangle \, \langle \phi_{0,k}, \cx \rangle + \sum_{i=1}^{2^n-1} \sum_{j=0}^N
\sum_{k\in I_{P,j}}\, \langle f, \psi_{i,j,k} \rangle\, \langle \psi_{i,j,k}, \cx\rangle \, ,
\]
where $J_P$ and $I_{P,j}$ are, respectively, as in  \eqref{menge1} and \eqref{menge2} with $Q$ replaced by $P$.
Now we look for sufficient conditions guaranteeing the existence of its limit as $N\to\fz$.
Clearly, the sum
$\sum_{k\in J_P}\, \langle f,\phi_{0,k} \rangle \, \langle \phi_{0,k}, \cx \rangle$ is a well-defined complex number.
Furthermore, because of the moment condition \eqref{moment}, we conclude that
\[
 \langle \psi_{i,j,k}, \cx\rangle = 2^{-jn/2}\, \int_{[0,2^{j}]^n}  \psi_{i} (x -k) \, dx = 0
\]
possibly except  those cases where
\begin{align}\label{ww-01}
\lf| \supp \psi_{i} (\, \cdot\,  -k) \cap [0,2^j]^n \right| \, \cdot\, \lf|\supp \psi_{i} (\, \cdot \,  -k)
\cap \lf(\rn \setminus [0,2^j]^n \right)\right| >0\, .
\end{align}
Let us denote the set of all such $k$ satisfying \eqref{ww-01} by $K_{i,j}$. Because of the compact support of our
generators of the wavelet system,
there exists  a finite positive constant
$c_1$, independent of $j \in \nn_0$, such that the cardinality of  $K_{i,j}$ is bounded by $c_1 2^{j(n-1)}$.
We fix a positive constant $c_2$ such that
\[
 c_2 \ge \max _{k \in \zz^n} \, \lf|\int_{[0,2^{j}]^n}  \psi_{i} (x -k) \, dx\right|,\quad \forall j\in\nn_0.
\]
Using these observations, we obtain
\begin{align}\label{estima}
\lf|
 \sum_{i=1}^{2^n-1} \sum_{j=0}^N
\sum_{k\in I_{P,j}}\, \langle f, \psi_{i,j,k} \rangle\, \langle \psi_{i,j,k}, \cx\rangle\right| \le
c_2 \,  \sum_{i=1}^{2^n-1} \sum_{j=0}^N  2^{-jn/2}\,
\sum_{k\in K_{i,j}}\, |\langle f, \psi_{i,j,k} \rangle| \, .
\end{align}
From the H\"older inequality for $p>1$, it follows that
\begin{align}\label{ww-02}
&\lf|
 \sum_{i=1}^{2^n-1} \sum_{j=0}^N
\sum_{k\in I_{P,j}}\, \langle f, \psi_{i,j,k} \rangle\, \langle \psi_{i,j,k}, \cx\rangle\right| \\
&\quad \le  c_1^{1/p'}\,
c_2 \, \sum_{i=1}^{2^n-1} \sum_{j=0}^N   2^{-jn/2}\, 2^{j(n-1)/p'} \, \lf(\sum_{k\in K_{i,j}}\, |\langle f, \psi_{i,j,k} \rangle|^p \right)^{1/p}.\nonumber
\end{align}
Observe that, if $k\in K_{i,j}$, then, because of the compact supports of the wavelets, there exists a natural number $D$, independent of $i$ and $j$, such that
\[
 Q_{j,k} \subset \bigcup_{|m|\le D} Q_{0,m}\, .
\]
This implies that
\begin{align*}
\sum_{i=1}^{2^n-1}&\sum_{j=0}^N   2^{-jn/2}\, 2^{j(n-1)/p'} \,
\lf(\sum_{k\in K_{i,j}}\, |\langle f, \psi_{i,j,k} \rangle|^p \right)^{1/p}\\
& \le
c_3 \, \max_{|m|\le D} \sum_{i=1}^{2^n-1} \sum_{j=0}^N   2^{-jn/2}\, 2^{j(n-1)/p'} \, \lf(\sum_{k: ~ Q_{j,k} \subset Q_{0,m}}\, |\langle f, \psi_{i,j,k} \rangle|^p \right)^{1/p}
\le  c_4 \, \|f\|_{B^{\frac 1p - 1, \tau}_{p,1} (\rn)};
\end{align*}
see Proposition \ref{wav-type2}. Thus, if $p\in[1, \infty]$,  the limit
$\lim_{N \to \infty}\, \langle S_Nf , \cx  \rangle$
exists for any $f \in B^{\frac 1p - 1, \tau}_{p,1} (\rn)$.

{\bf Step 2)}  Proof of (ii). Let $p\in(0, 1)$.
Obviously $f \mapsto \langle f , \cx  \rangle$  makes sense for any $f \in  L^1(\rn)$. Thus, to show $\langle f, \cx\rangle$
can be extended to $\bt$, it suffices to prove the embedding $\bt\subset L^1(\rn)$.
It is also clear that we may assume $\supp f \subset [-1,2]^n$,
because smooth functions with compact supports are pointwise multipliers on $\bt$ (see \cite[Theorem 6.1]{ysy}).

Now we prove that we will obtain a sufficient condition
by studying the embedding $\bt \hookrightarrow L^{1}(\rn)$.
In \cite[Theorem~3.8(i)]{HMS16},
Haroske et al. showed that $\bt \subset L^1_{\loc} (\rn)$ when $p\in(0, 1)$,
$s= (1-p \, \tau )n(\frac 1p - 1)$ and $q\in (0,p]$.
Looking into the details of their proof, we find that one can sharpen their result as follows:
there exists a positive constant $c$ such that
\[
 \int_{[-1,2]^n} |f(x)|\, dx \le c \, \|f\|_{\bt}
\]
for any function $f\in \bt$ with support contained in $[-1,2]^n$.  Thus,
$\langle f , \cx  \rangle$  makes sense for any $f \in \bt$ when $p\in(0, 1)$,
$s= (1-p \, \tau )n(\frac 1p - 1)$ and $q\in (0,p]$.

{\bf Step 3)} Proof of (iii).
We argue as in Step 2) but using
the continuous embeddings
\[
B^{\frac np - n,0}_{p,1} (\rn) = B^{\frac np - n}_{p,1} (\rn) \hookrightarrow B^{0,0}_{1,1} (\rn)= B^{0}_{1,1} (\rn)
\hookrightarrow L^1(\rn)\, , \qquad \forall\,p\in(0,1)\, ;
\]
see, for instance, \cite{ST}.
The proof of Theorem \ref{limit2} is then complete.
\end{proof}

We now consider another variant of extending  $f \mapsto \langle f , \cx  \rangle$  to $B^{s, \tau}_{p,q} (\rn)$.

\begin{theorem}\label{limit2-}
Let $s\in\rr$, $p,\,q\in(0, \infty]$ and $\tau \in [0,\infty)$.
If $s>n/p-n\tau-1$,
then $f \mapsto \langle f , \cx  \rangle$  extends to a continuous linear functional on $B^{s, \tau}_{p,q} (\rn)$
which coincides on $B^{s, \tau}_{p,q} (\rn) \cap L^r(\rn)$, $r \in[1, \infty]$, with $\int_\rn f (x)  \cx (x) \, dx$.
\end{theorem}

\begin{proof}
We only need to modify the proof of Theorem \ref{limit2}
after \eqref{estima}. Using
Proposition \ref{wav-type2}, we know that
\[
 |\langle f, \psi_{i,j,k} \rangle|\ls 2^{-j(s+n\tau+\frac n2-\frac np)}\|f\|_{\bt}
\]
holds true for any $i\in\{1,\ldots,2^n-1\}$, $j\in\nn_0$, $k\in\zz^n$ and $f \in \bt$.
Using  \eqref{estima}, we conclude, with $s>n/p-n\tau-1$, that
\begin{align*}
&\lf|
 \sum_{i=1}^{2^n-1} \sum_{j=0}^N
\sum_{k\in I_{P,j}}\, \langle f, \psi_{i,j,k} \rangle\, \langle \psi_{i,j,k}, \cx\rangle\right|\\
&\quad\ls  \sum_{i=1}^{2^n-1} \sum_{j=0}^N  2^{-jn/2}\,
 2^{j(n-1)} 2^{-j(s+n\tau+\frac n2-\frac np)}\|f\|_{\bt}\ls \|f\|_{\bt},
\end{align*}
which completes the proof of Theorem \ref{limit2-}.
\end{proof}


\subsection{Necessary conditions for the existence of $ \langle f , \cx  \rangle$\label{s5.2}}


Next we turn to negative results concerning the existence of the limit in \eqref{limit1}.
Therefore we will construct several  families of  test functions.

To prove these negative statements we need several specific properties of the generators of our
wavelet system.
First, we require to work with a wavelet system which is admissible in the sense of Proposition \ref{wav-type2}
for $B^{s,\tau}_{p,q} (\rn)$ with $p$, $q$, $\tau$ fixed and
\[
\min\lf\{\frac np - n \tau -1, \max \lf( n\lf[ \frac 1p-1\right], \frac 1p-1\right)\right\}-1 \le s
\le \min \lf\{\frac 1p, n\lf(\frac 1p - \tau\right)\right\} + 1\, .
\]
In addition we require that it will be of Daubechies type (see \cite[4.1,~4.2]{woj}).
Let $\widetilde{\phi}$ be a scaling function and $\widetilde{\psi}$ an associated wavelet (both on $\rr$).
Here we have the possibility to work also with a shift of these two functions without changing the wavelet system.
We claim the following: we may choose a generator $\widetilde{\psi} \in C^{N_1}(\rr)$ such that:
\begin{itemize}
 \item[(a)] There exist integers $K<0$ and $L>0$ such that
 \[
  \supp \widetilde{\psi} \subset[K,L]\, .
 \]
\item[(b)] There exists a natural number $j_0$ such that $2^{-j_0}L \le 1$ and
\begin{align}\label{intpsi}
 \int_0^1 \widetilde{\psi} (2^{j_0}t)\, dt \neq 0\, .
\end{align}
\end{itemize}
Part (a) can be found in \cite[4.1]{woj}.
To show (b), we argue by contradiction, namely, we assume that, for any generator
$\widetilde{\psi} (\, \cdot \, - m)$, $m \in \zz$, of this wavelet system and for any $j \in \nn$ satisfying $L\le 2^j$,
we have
\[
 \int_0^1 \widetilde{\psi} (2^{j}t-m)\, dt =  0\,  .
\]
Then, in the one-dimensional  case, it follows
\[
\cx   =  \sum_{k\in\zz}\, \langle \cx, \widetilde\phi_{0,k} \rangle \, \widetilde\phi_{0,k} +  \sum_{j=0}^{j_0-1}
\sum_{k\in\zz}\, \langle \cx,\widetilde\psi_{j,k} \rangle\, \widetilde\psi_{j,k}\, ,
\]
where $j_0$ is as in \eqref{intpsi}, $\widetilde\phi_{0,k}(x):=\widetilde\phi(x-k)$ and $\widetilde\psi_{j,k}(x):=2^{j/2}\widetilde\psi(2^jx-k)$
for any $x\in\rr$, $j\in\nn_0$ and $k\in\zz$.
Since the summations on the right-hand side of the above formula are all finite, we conclude by  Proposition \ref{wav-type2}
that $\cx$ belongs to $B^{s,\tau}_{p,q} (\rr)$ for any $s \le \frac 1p - \tau + 1 $.
But this is in contradiction with Theorem \ref{charact}.
Thus, \eqref{intpsi} holds for some $j_0$. Without loss of  generality, we may assume that the integral in \eqref{intpsi}
 is \emph{positive}.

In addition we choose the  scaling function $\widetilde{\phi}\in C^{N_1}(\rr)$ (again the degree of freedom we have is the shift)
such that $\supp\wz{\phi}\subset [0, \widetilde{L}]$ with $\widetilde{L} := L-K$ and
\[
\int_{0}^{\widetilde{L}} \wz \phi (t)\, dt  > 0
\]
(see \cite[4.1,~4.2]{woj}).
Let
\begin{align}\label{wapsi1}
\psi_1 (x) := \widetilde{\psi}(x_1)  \widetilde{\phi}(x_2) \cdots   \widetilde{\phi}(x_n)
\, , \qquad \forall\,x:= (x_1, x_2,\ldots , x_n) \in \rn\, .
\end{align}
Let $j_0$ be defined as in (b) and let $\{\lambda_\ell\}_{\ell=j_0}^\infty$
be a given sequence  of real numbers.
We define  our first family of test functions as
\begin{align}\label{test}
f_N(x) :=   \sum_{\ell=j_0}^N  \lambda_{\ell} \sum_{\gfz{k=(0,k_2, \ldots , k_n)}
{0 \le k_i < 2^{\ell} - \widetilde{L}, ~ i\in\{2, \ldots, n\}
}} \, \psi_{1,\ell,k} (x) \, , \quad \forall\,x \in \rn, \quad \forall\, N \in \nn\cap [j_0,\fz),
\end{align}
where $\psi_{1,\ell,k}$ is as in \eqref{3.4x} with $i=1$ and $j=\ell$.
Clearly, $f_N$ is  as smooth as the elements of the wavelet system and has compact support.
More exactly, the support is concentrated near a part of the boundary of $[0,1]^n$.
For us important are the following estimates of   $\|f_N\|_{\bt}$.

\begin{lemma}\label{limit3}
Let $s\in\rr$, $p,\,q\in(0,\infty]$ and  $\tau \in[0,\infty)$.
Then, for any $N\in\nn\cap [j_0,\fz)$,
\begin{align}\label{limit4}
 \|f_N \|_{B^{s, \tau}_{p,q} (\rn)}  \asymp  \left\{\begin{array}{lll}
\left\{\sum\limits_{j  = j_0}^{N}
2^{j (s+\frac n2-\frac 1p)q} \, |\lambda_{j}|^q \right\}^{1/q} && \mbox{if}\  \tau \in\Big[0, \frac{n-1}{n} \, \frac 1p\Big],
\\
\displaystyle\sup_{J\in\{j_0, \ldots , N\}}\, 2^{J(n\tau - \frac{n-1}{p})} \,
\left\{\sum\limits_{j  = J}^{N}
2^{j (s+\frac n2-\frac 1p)q} \, |\lambda_{j}|^q \right\}^{1/q} &&\mbox{if}\  \tau\in\lf(\frac{n-1}{n} \, \frac 1p,\fz\right)
\end{array}\right.
\end{align}
with the positive  equivalence constants independent of $N$ and $\{\lambda_j\}_{j=j_0}^\fz$.
\end{lemma}

\begin{proof}
To estimate  $\|f\|_{B^{s,\tau}_{p,q} (\rn)}$, we proceed by using  Proposition \ref{wav-type2}.
Here the following observations will be applied:
\begin{itemize}
 \item All $Q_{j,k}$,  associated to a non-zero coefficient $\lambda_j$, are subsets of $[0,1]^n$.
Thus, it will be enough to consider $P \subset Q_{0,0}:= [0,1]^n$.
\item If $P:= Q_{j,k} \subset [0,1]^n$, then it is enough to consider those $k$
such that $P \cap \partial [0,1]^n \neq \emptyset$, where $\partial[0,1]^n$ denotes the \emph{boundary} of $[0,1]^n$.

\item For $J\in \nn_0$ fixed, the cube $ Q_{J,0}$ leads to the largest contribution, more exactly,
\[
 \max_{k \in \zz^n} \, \sum_{\{m:~ Q_{j,m} \subset Q_{J,k}\}} |\langle f_N,\,\psi_{i,j,m}\rangle|^p\le  \sum_{\{m:~ Q_{j,m} \subset Q_{J,0}\}} |\langle f_N,\,\psi_{i,j,m}\rangle|^p\, .
\]
\item If $J > N $, then there exists no cube $ Q_{j,m}$, associated to a non-zero coefficient and  contained in $P:= Q_{J,k}$.
\end{itemize}

Then, by the orthogonality of the wavelet system, we have
\begin{align*}
\|f_N \|_{B^{s, \tau}_{p,q} (\rn)}
&\asymp   \sup_{P\in\mathcal{Q}}\frac1{|P|^{\tau}}
\,
\left\{\sum_{j  =  \max \{j_P,j_0\}}^{N}
2^{j(s+\frac n2-\frac np)q} \,  \lf( \sum_{\{\{m:~ Q_{j,m} \subset P\}\}} |\langle f_N,\,\psi_{1,j,m}\rangle|^p\right)^{q/p}\right\}^{1/q}
\\
& \asymp    \sup_{J\in\{j_0,\ldots, N\}}\, 2^{Jn\tau} \,
\left\{\sum_{j  =  J}^{N}
2^{j(s+\frac n2-\frac np)q} \,  \lf( \sum_{\{m:~ Q_{j,m} \subset Q_{J,0}\}} |\langle f_N,\,\psi_{1,j,m}\rangle|^p\right)^{q/p}\right\}^{1/q}
\\
& \quad +     \sup_{J\in\{0,\ldots, j_0 -1\}}\, 2^{Jn\tau} \,
\left\{\sum_{j  =  j_0}^{N}
2^{j(s+\frac n2-\frac np)q} \,  \lf( \sum_{\{m:~ Q_{j,m} \subset Q_{J,0}\}} |\langle f_N,\,\psi_{1,j,m}\rangle|^p\right)^{q/p}\right\}^{1/q}
\\
&=: S^- + S^+ \, .
\end{align*}
Checking the number of cubes $Q_{j,m} \subset Q_{J,0}$ with $m:= (0,m_2, \ldots \, , m_n)$, we find that
\begin{align*}
S^+ & \asymp
\sup_{J\in\{0,\ldots, j_0 -1\}}\, 2^{Jn\tau} \,
\left\{\sum_{j  =  j_0}^{N}
2^{j(s+\frac n2-\frac np)q} \, |\lambda_{j}|^q 2^{(j - J)(n-1)q/p}\right\}^{1/q}
\\
& \asymp  \sup_{J\in\{0,\ldots, j_0 -1\}}\, 2^{J ( n\tau - \frac{n-1}{p})} \,
\left\{\sum_{j  =  j_0}^{N}
2^{j(s+\frac n2-\frac 1p)q} \, |\lambda_{j}|^q \right\}^{1/q} \, .
\end{align*}
Similarly,
\begin{align*}
S^- & \asymp
 \sup_{J\in\{j_0,\ldots, N\} }\, 2^{Jn\tau} \,
\left\{\sum_{j  =  J}^{N}
2^{j(s+\frac n2-\frac np)q} \,  |\lambda_{j}|^q 2^{(j - J)(n-1)q/p}\right\}^{1/q}
\\
& \asymp   \sup_{J\in\{j_0, \ldots, N\}}\, 2^{J(n\tau - \frac{n-1}{p})} \,
\left\{\sum_{j  = J}^{N}
2^{j (s+\frac n2-\frac 1p)q} \, |\lambda_{j}|^q \right\}^{1/q}
\, .
\end{align*}
Now we have to distinguish two cases: $0 \le \tau \le \frac{n-1}{n} \, \frac 1p$ and $\frac{n-1}{n} \, \frac 1p < \tau$.
In the first case we obtain

\[
 \|f_N \|_{B^{s, \tau}_{p,q} (\rn)} \asymp
\left\{\sum_{j  = j_0}^{N}
2^{j (s+\frac n2-\frac 1p)q} \, |\lambda_{j}|^q \right\}^{1/q}\, ,
\]
whereas in the second case we conclude that

\[
 \|f_N \|_{B^{s, \tau}_{p,q} (\rn)} \asymp
\sup_{J\in\{j_0, \ldots , N\}}\, 2^{J(n\tau - \frac{n-1}{p})} \,
\left\{\sum_{j  = J}^{N}
2^{j (s+\frac n2-\frac 1p)q} \, |\lambda_{j}|^q \right\}^{1/q} \, .
\]
This finishes the proof of Lemma \ref{limit3}.
\end{proof}

Now we turn to the calculation of  $\langle f_N, \cx\rangle$.

\begin{lemma}\label{scalar1}
For any $N\in \nn\cap [j_0,\fz)$, it holds
\begin{align}\label{test2}
\langle f_N, \cx\rangle
\asymp
 \sum_{j=j_0}^N  \lambda_{j} \,  2^{j( \frac n2 -1)}
\end{align}
with the positive equivalence  constants independent  of $N$ and $\{\lambda_j\}_{j=j_0}^\fz$, where $j_0$ is as in \eqref{intpsi}.
\end{lemma}

\begin{proof}
By the definition of $f_N$ and the choice of $j_0$, we conclude that, for any $N\in\nn\cap[j_0,\fz)$,
\begin{align}\label{test1}
\langle f_N, \cx\rangle
& = \int_\rn \lf[\sum_{j=j_0}^N  \lambda_{j} \,
\sum_{\gfz{k=(0,k_2, \ldots, k_n)}{0 \le k_i < 2^{j}-\widetilde{L}, ~ i=2, \ldots \, , n}}
\psi_{1,j,k} (x)\right] \, \cx (x) \,dx\nonumber
\\
&  =    \sum_{j=j_0}^N  \lambda_{j} \, 2^{j n/2}\,  \int_0^1 \psi_{1}  (2^{j} x_1) \, dx_1
\sum_{\gfz{k=(0,k_2, \ldots \, , k_n)}{0 \le k_i < 2^{j} - \widetilde{L}, ~ i=2, \ldots, n}}
\prod_{i=2}^n \int_0^1 \widetilde{\phi}  (2^{j}x_i-k_i) \, dx_i
\nonumber
\\
& =   \sum_{j=j_0}^N  \lambda_{j} \, 2^{j n/2} \, \lf[\int_0^{2^{-j}L} \psi_{1}  (2^{j} t) \, dt\right]
\sum_{\gfz{k=(0,k_2, \ldots \, , k_n)}{0 \le k_i < 2^{j}- \widetilde{L}, ~ i=2, \ldots, n}} 2^{-j(n-1)}
\prod_{i=2}^n \int_{-k_i}^{2^j-k_i} \widetilde{\phi}  (x_i) \, dx_i\, .
\nonumber
\end{align}
Employing our assumption concerning the support of $\widetilde{\phi}$ and the definition of $j_0$,
we obtain, for any $N\in\nn\cap[j_0,\fz)$,
\begin{align*}
\langle f_N, \cx\rangle
& =  \sum_{j=j_0}^N  \lambda_{j} \, 2^{j n/2} \, \lf[\int_0^{2^{-j}L} \psi_{1}  (2^{j} t) \, dt\right]
\sum_{\gfz{k=(0,k_2, \ldots \, , k_n)}{0 \le k_i < 2^{j} - \widetilde{L}, ~ i=2, \ldots, n}} 2^{-j(n-1)}
\prod_{i=2}^n \int_{0}^{L} \widetilde{\phi}  (x_i) \, dx_i
\\
& \asymp  2^{j_0}
 \sum_{j=j_0}^N  \lambda_{j} \, \, 2^{j ( \frac n2 -1)} \, \lf[\int_0^{2^{-j_0}L} \psi_{1}  (2^{j_0} t) \, dt\right]
\asymp \sum_{j=j_0}^N  \lambda_{j} \,  2^{  j( \frac n2 -1)}.
\end{align*}
This proves the desired result of Lemma \ref{scalar1}.
\end{proof}


To deal with the case   $p\in(0,1)$, we   need a second sequence of test functions.
Therefore we need a preparation.
Let $\alpha \in (0,n)$ and  $j \in \nn$.
We would like to distribute $\lfloor 2^{j\alpha} \rfloor$ dyadic cubes
$Q_{j,m}$  in  $Q_{0,0}:=[0,1)^n$ in a rather specific way (not uniformly). We claim that there exists
a set $A_j \subset \nn_0^n$ of cardinality $\lfloor 2^{j\alpha} \rfloor$
such that
\[
Q_{j,m} \subset [0,1]^n  \qquad \mbox{for any} \quad m \in A_j\, ,
\]
\begin{align}\label{vor1}
\max_{\gfz{K=(K_1, \ldots \, , K_n)}{0 \le K_i < 2^J, ~ i=1, \ldots \, , n}}
\sum_{\{m\in A_j:~Q_{j,m}\subset Q_{J,K}\}} 1 \le 2\,  \lf(\sum_{\{m\in A_j:~Q_{j,m}\subset Q_{J,0}\}} 1\right) \, , \qquad\forall\, J \in\{0,\ldots,j\},
\end{align}
and  there exists a positive constant $\wz C$, independent of $j \in \nn$, such that
\begin{align}\label{vor2}
 \, \lfloor 2^{(j-J) \alpha} \rfloor \le  \sum_{\{m\in A_j:~Q_{j,m}\subset Q_{J,0}\}} 1 \le \wz C \, 2^{(j-J)\alpha} \, ,  \qquad\forall\, J \in\{0,\ldots,j\}.
\end{align}

The construction of such an $A_j$ is rather easy.
Clearly
\[
 [0,1)^n = Q_{j,0} \cup \lf[\bigcup_{J=1}^{j} (Q_{J-1,0}\setminus Q_{J,0})\right]\, .
\]
We will define $A_j$ via first defining the subset  of $A_j$ in  each $Q_{J-1,0}\setminus Q_{J,0}$ with $J\in\{1, \ldots  ,j\}$.
First, we put $0$ into the set $A_j$, which means that $Q_{j,0}$ is chosen.
Inside $Q_{j-1,0}\setminus Q_{j,0}$, there exist $2^n-1$ dyadic cubes in $\cq_j$ and we need to select
$\lfloor 2^{\alpha}-1 \rfloor$. Which one we take is unimportant.
We proceed by induction with induction hypothesis
\begin{align}\label{estim}
\lfloor 2^{(j-J)\alpha} \rfloor \le \lf|\lf\{m \in A_j: ~~Q_{j,m} \subset Q_{J,0} \right\} \right| \le \wz C \,  2^{(j-J)\alpha} \, ,
\end{align}
where $\wz C \ge 1$ is a constant independent of $j$ and $J$ and  will be determined later.
It will be sufficient to look for the step from $J$ to $J-1$.
Of course there exist $2^n-1$ cubes $Q_{J,K}$ such that
\[
 Q_{J-1,0} = Q_{J,0} \cup \lf(\bigcup_{K} Q_{J,K}\right)\, .
\]
Altogether we have $(2^{n}-1)\, 2^{(j-J)n}$ cubes $Q_{j,m}$ in $Q_{J-1,0}\setminus Q_{J,0}$.
We decide for an almost uniform distribution, namely,
in each $Q_{J,K} \subset (Q_{J-1,0}\setminus Q_{J,0})$, we select
\[
\lf\lfloor \frac{2^{(j-J+1)\alpha}  -  2^{(j-J)\alpha}}{2^n-1}\right\rfloor + 1
\]
cubes $Q_{j,m}$.
This guarantees the lower bound  for
$|\{m \in A_j: ~~Q_{j,m} \subset Q_{J-1,0} \}|$ as in \eqref{estim} with $J$ replaced by $ J-1$.
Now we deal with the corresponding upper bound. Notice that
\begin{align*}
\wz C \,  2^{(j-J)\alpha} & +
 (2^n-1)\lf\{\lf\lfloor \frac{2^{(j-J+1)\alpha}  -  2^{(j-J)\alpha}}{2^n-1}\right\rfloor + 1 \right\}
\\
&\le
\wz C \,  2^{(j-J)\alpha}  +
 (2^\alpha-1)\,  2^{(j-J)\alpha} + 2^n -1 \, .
\end{align*}
With
\[
 \wz C\ge \frac{2^n + 2^\alpha -2}{2^\alpha -1}\,
\]
we conclude
\[
 \wz C \,  2^{(j-J)\alpha}  +
 (2^n-1)\lf(\lf\lfloor \frac{2^{(j-J+1)\alpha}  -  2^{(j-J)\alpha}}{2^n-1}\right\rfloor + 1 \right)
\le \wz C\,  2^{(j-J+1)\alpha}, \quad\forall\, J\in\{0,\ldots, j-1\},
\]
which implies that  $|\{m \in A_j: ~~Q_{j,m} \subset Q_{J-1,0} \}|$ satisfies the upper bound in \eqref{estim}
with $J$ replaced by $J-1$.

Since \eqref{estim} is just \eqref{vor2}, we know that
\eqref{vor2} is fulfilled by $A_j$ determined as above.
On another hand, since
\[
\lf\lfloor \frac{2^{(j-J+1)\alpha}  -  2^{(j-J)\alpha}}{2^n-1}\right\rfloor + 1 \le
\lfloor2^{(j-J)\alpha}\rfloor + 1 \le 2\, \lfloor 2^{(j-J)\alpha} \rfloor \, , \qquad \forall\,J\in\{0,\ldots,j\},
\]
one also finds  that \eqref{vor1} is fulfilled. This proves the previous claim. Let us mention that our construction was inspired
by a similar one in \cite[Proof of Theorem~3.1, Substep 2.4]{HS13}.

Below we need to indicate the dimension $n$ in which our construction took place.
So we will use the notation $A^{n}_j$ instead of $A_j$.

Now we are in position to introduce our second family of test functions.
For fixed $\az\in(0,n-1)$, $j_1 \ge j_0$ and a given sequence$\{\lambda_j\}_{j=j_1}^\infty$
of real numbers, we define
$$T_j:=
\lf\{ m:=(0,m_2, \ldots , m_n): ~(m_2, \ldots \, , m_n) \in A^{n-1}_j,\   \max\{m_2,\ldots,m_n\} <2^j-\widetilde{L} \right\}$$
for any $j\in \{j_1,j_1+1,\ldots\}$, and
\begin{align}\label{folge3}
 g_N (x):= \sum_{j=j_1}^N  \lambda_{j} \sum_{m\in T_j}
\psi_{1,j,m} (x), \quad \forall\,x \in \rn,\ \ \forall\,N\in\nn\cap[j_1,\fz),
\end{align}
which $A^{n-1}_j \subset \nn_0^{n-1}$ is a set constructed as  $A_j=A_j^n$ above but with $n$ replaced by $n-1$.

As $f_N$, the function $g_N$ is as smooth as the generators of the wavelet system and it has compact support.
We have the following estimate.

\begin{lemma}\label{limit6}
Let $s\in\rr$, $p,\,q\in(0,\infty]$ and  $\tau \in (0,\frac{n-1}{np})$. If $\az=np\tau$, then,
for any $N\in\nn\cap[j_1,\fz)$,
\begin{align}\label{limit7}
 \|g_N \|_{B^{s, \tau}_{p,q} (\rn)}  \asymp  \left\{\sum\limits_{j  = j_1}^{N}
2^{j (s+\frac n2 + n\tau -\frac np)q} \, |\lambda_{j}|^q \right\}^{1/q}
\end{align}
with the positive equivalence  constants independent of $N$ and $\{\lambda_j\}_{j=j_1}^\fz$, where $j_1\ge j_0$ and $j_0$
is as in \eqref{intpsi}.
\end{lemma}

\begin{proof}
First, observe that
$\alpha=np\tau < n-1$ as required in our previous construction.
As in the case of   $f_N$, we conclude that
\begin{align*}
\|g_N \|_{B^{s, \tau}_{p,q} (\rn)}
& \asymp    \sup_{J\in\{j_1,\ldots, N\}}\, 2^{Jn\tau} \,
\left\{\sum_{j  =  J}^{N}
2^{j(s+\frac n2-\frac np)q} \,  \lf( \sum_{\{m:~ Q_{j,m} \subset Q_{J,0}\}} |\langle g_N,\,\psi_{1,j,m}\rangle|^p\right)^{q/p}\right\}^{1/q}
\\
& \quad +     \sup_{J\in\{0,\ldots, j_1 -1\}}\, 2^{Jn\tau} \,
\left\{\sum_{j  =  j_1}^{N}
2^{j(s+\frac n2-\frac np)q} \,  \lf( \sum_{\{m:~ Q_{j,m} \subset Q_{J,0}\}} |\langle g_N,\,\psi_{1,j,m}\rangle|^p\right)^{q/p}\right\}^{1/q}
\\
&=:  S^- + S^+ \, .
\end{align*}
Using the properties of $A_j^{n-1}$, namely, \eqref{vor1} and \eqref{vor2} with $n$ replaced by $n-1$, we find that
\begin{align*}
S^+ & \asymp
\sup_{J\in\{0,\ldots, j_1 -1\}}\, 2^{Jn\tau} \,
\left\{\sum_{j  =  j_1}^{N}
2^{j(s+\frac n2-\frac np)q} \, |\lambda_{j}|^q 2^{(j - J)\alpha q/p}\right\}^{1/q}
\\
& \asymp  \sup_{J\in\{0,\ldots, j_1 -1\}}\, 2^{J ( n\tau - \frac{\alpha}{p})} \,
\left\{\sum_{j  =  j_1}^{N}
2^{j(s+\frac n2 + \frac \alpha p-\frac np)q} \, |\lambda_{j}|^q \right\}^{1/q}
 \asymp
\left\{\sum_{j  =  j_1}^{N}
2^{j(s+\frac n2 + n\tau -\frac np)q} \, |\lambda_{j}|^q \right\}^{1/q} \, .
\end{align*}
Similarly,
\begin{align*}
S^- & \asymp
 \sup_{J\in\{j_1,\ldots, N\}}\, 2^{Jn\tau} \,
\left\{\sum_{j  =  J}^{N}
2^{j(s+\frac n2-\frac np)q} \,  |\lambda_{j}|^q 2^{(j - J)\alpha q/p}\right\}^{1/q}\\
  &\asymp   \, \left\{\sum_{j  = j_1}^{N}
2^{j (s+\frac n2 + n \tau -\frac np)q} \, |\lambda_{j}|^q \right\}^{1/q}
\, .
\end{align*}
This proves our desired estimate of Lemma \ref{limit6}.
\end{proof}

The parameter $j_1$ does not play a role in Lemma \ref{limit6},
but it will be used for the next result of the estimate of $\langle g_N, \cx\rangle$.

\begin{lemma}\label{scalar2}
Let $j_1 \ge j_0$ be sufficiently large with $j_0$ as in \eqref{intpsi}. Then, for any $N\in\nn\cap[j_1,\fz)$,
\begin{align}\label{test5}
\langle g_N, \cx\rangle
\asymp
 \sum_{j=j_1}^N  \lambda_{j} \,  2^{ j(\alpha - \frac n2)}
\end{align}
with the positive equivalence   constants independent  of sufficiently large $N$ and $\{\lz_j\}_{j=j_1}^\fz$.
\end{lemma}

\begin{proof}
Since the cardinality of $A^{n-1}_j$ equals $\lfloor 2^{j\alpha}\rfloor$, we know
the cardinality  $|T_j|<  2^{j\alpha}$  for any $j\in\{j_1,j_1+1,\ldots\}$.

On another hand, let $j_1$ be large enough such that $2^{j_1}>\widetilde L$.
For each $j\ge j_1$, let $J_j\in\{1,\ldots, j\}$  be the unique number satisfying  $2^{j-J_j}\le 2^{j}-\widetilde L< 2^{j-J_j+1}$.
Then,
by the property of $A^{n-1}_j$ [\eqref{vor2} with $n$ replaced by $n-1$], we conclude that
\begin{align*}
|T_j|&\ge \lf|\lf\{ m:=(0,m_2, \ldots , m_n): ~(m_2, \ldots \, , m_n) \in A^{n-1}_j,\   \max\{m_2,\ldots,m_n\} <2^{j-J_j}\right\}\right|\\
&\ge \lfloor 2^{(j-J_j)\az} \rfloor \ge 2^{(j-J_j)\az}-1 >2^{-\az} \lf(2^j-\widetilde L\right)^\az -1.
\end{align*}
Altogether we obtain
$$2^{-\az}\lf(2^j-\widetilde L\right)^\az -1<|T_j|<  2^{j\alpha}.$$
From this and an argument similar to that used in the proof of Lemma \ref{scalar1}, we find that
\eqref{test5} holds  if
 \[
  2^{-\az}\lf(2^j-\widetilde L\right)^\az -1 \asymp 2^{j\alpha} \, ,
 \]
which is true when $j\in\{j_1,j_1+1,\ldots\}$ and $j_1$ is sufficiently large. This proves Lemma \ref{scalar2}.
\end{proof}


\subsection{Proofs of Theorems \ref{limit5} and \ref{limit5b}\label{s5.3}}


We are now ready to prove Theorems \ref{limit5} and \ref{limit5b}.

{\bf Step 1)} Sufficiency.
The  {\em if}-cases in Theorems \ref{limit5} and \ref{limit5b} are  covered by Theorems \ref{limit2}
and \ref{limit2-}.

{\bf Step 2)} Necessity. To show that $\langle f,\cx\rangle$ cannot be extended to a Besov-type space $\bt$,
it suffices to find a  sequence $\{f_N\}_{N}$ of functions in  $\bt$
such that $\{|\langle f_N,\cx\rangle|/\|f_N\|_{\bt}\}_{N}$ is unbounded.
Below we only concentrate on limiting cases (if there is one), because
for the remaining parameter constellations, one  can use the elementary embedding mentioned in Remark \ref{grund}.

{\bf Substep 2.1)}
Let $p\in[1,\fz]$, $s= \frac 1p - 1$, $q\in(1,\infty]$ and $\tau\in[0,\frac{n-1}{np}]$.
We choose a sequence of positive real numbers, $\{\mu_j\}_{j=j_0}^\infty$, such that
\[
 \sum_{j=j_0}^\infty  \mu_{j}  =\infty \qquad \mbox{and} \qquad  \lf\{\sum\limits_{j  = j_0}^{\infty} \, \mu_{j}^q \right\}^{1/q}<\infty\,,
\]
where $j_0$ is as in \eqref{intpsi}.
Now let $\lambda_j := 2^{-j(\frac n2 -1)} \mu_j$ for any $j \in\{j_0,j_0+1,\ldots\}$.
Then the associated sequence $\{f_N\}_{N=j_0}^\fz$, defined in \eqref{test}, is bounded in
$B^{\frac 1p -1, \tau}_{p,q} (\rn)$ (see Lemma \ref{limit3}), while
$\{\int_\rn f_N(x)  \cx (x) \, dx\}_{N=j_0}^\fz$ is unbounded (see Lemma \ref{scalar1}).
Thus, in this case, $\{f_N\}_{N=j_0}^\fz$ is the desired sequence.

{\bf Substep 2.2)}
Let $p\in (0,\fz]$, $s= \frac np - n \tau - 1$, $q\in(0,1]$ and $\tau\in (\frac{n-1}{np},\fz)$.
Define $\lz_j:=j2^{-j(\frac n2-1)}$ for any $j\in\nn_0$, and let $\{f_N\}_{N=j_0}^\fz$ be as in \eqref{test}.
Then Lemma \ref{scalar1} tells us that, for any $N\in\{j_0,j_0+1,\ldots\}$,
$$
|\langle f_N,\cx\rangle| \gtrsim  \sum_{j=j_0}^N  j \asymp \frac{N^2-j_0^2}2.$$
Meanwhile, \eqref{limit4} gives that, for any $N\in\{j_0,j_0+1,\ldots\}$,
\begin{align*}
 \|f_N \|_{B^{\frac np-n\tau-1, \tau}_{p,q} (\rn)}
&\asymp
 \sup_{J\in\{j_0, \ldots , N\}}\, 2^{J(n\tau - \frac{n-1}{p})} \,
\left\{\sum\limits_{j  = J}^{N}
2^{j (\frac {n-1}p-n\tau)q} \, j^q \right\}^{1/q}.
\end{align*}
Since $\frac {n-1}p-n\tau < 0$, it follows that, for any $J\in\{j_0,j_0+1,\ldots\}$ and $N\in\{J, J+1,\ldots\}$,
\[
\sum_{j  = J}^{N} 2^{j (\frac {n-1}p-n\tau)q} \, j^q \asymp  J^q 2^{J(\frac{n-1}p-n\tau)q}.
\]
This implies that, for any $N\in\{j_0,j_0+1,\ldots\}$,
\begin{align*}
 \|f_N \|_{B^{\frac np-n\tau-1, \tau}_{p,q} (\rn)}
&\asymp
 \sup_{J\in\{j_0, \ldots , N\}}\, 2^{J(n\tau - \frac{n-1}{p})} J 2^{J(\frac{n-1}p-n\tau)}\asymp N.
\end{align*}
Therefore,
$$\frac{|\langle f_N,\cx\rangle|}{\|f_N\|_{B^{\frac np-n\tau-1, \tau}_{p,q} (\rn)}} \to \fz$$
as $N\to \fz$. Thus, in this case, $\{f_N\}_{N=j_0}^\fz$ is the desired sequence.

{\bf Substep 2.3)} Let $ p\in(0, 1)$, $ q \in(1, \infty]$, $\tau \in (0, \frac{n-1}{np})$
and $s= (1-\tau p)n(\frac 1p -1)$.
We shall work with the family $\{g_N\}_{N=j_1}^\fz$ defined in \eqref{folge3} with $\az=pn\tau$.
Using Lemma \ref{limit6}, we obtain, for any $N\in\{j_1,j_1+1,\ldots\}$,
\begin{align}\label{limit8}
 \|g_N \|_{B^{s, \tau}_{p,q} (\rn)}  \asymp  \left\{\sum\limits_{j  = j_1}^{N}
2^{j (-\frac n2 + pn\tau)q} \, |\lambda_{j}|^q \right\}^{1/q}\, .
\end{align}
On another hand, Lemma \ref{scalar2} implies that, for any $N\in\{j_1,j_1+1,\ldots\}$,
\begin{align}\label{test6}
\langle g_N, \cx\rangle
\asymp
 \sum_{j=j_1}^N  \lambda_{j} \,  2^{ j(pn\tau  - \frac n2)}\, .
\end{align}
If $\{\mu_j\}_{j=j_1}^\fz$ is
the same as in Substep 2.1) with $j_0$ therein replaced by $j_1$,
and if we choose $\lambda_j$ such that $\lambda_j
:=  2^{-j(pn\tau  - \frac n2)}\mu_{j} $ for any $j\in\{j_1,j_1+1,\ldots\}$,
one finds that  $\{|\langle g_N,\cx\rangle|/\|g_N\|_{\bt}\}_{N=j_1}^\fz$ is unbounded. Thus, in this case, $\{g_N\}_{N=j_1}^\fz$ is the desired sequence.

{\bf Substep 2.4)} Let $p\in(0,1)$, $q\in(1, \infty]$, $\tau = \frac{n-1}{np}$ and $s= \frac 1p -1$.
Here we can employ the same example as in Substep 2.1) to obtain the desired sequence.

{\bf Substep 2.5)} Let $p\in(0, 1)$, $q \in(1, \infty]$, $\tau = 0$ and $s= n(\frac 1p -1)$.
We choose
\begin{align}\label{test9}
h_N(x) :=   \sum_{j=j_0}^N  \lambda_{j} \, \psi_{1,j,0} (x) \, , \qquad \forall\,x \in \rn, \quad \forall\,N \in\{j_0,j_0+1,\ldots\}.
\end{align}
As in Lemma \ref{limit3}, we have, for any $N\in\{j_0,j_0+1,\ldots\}$,
\[
 \|h_N \|_{B^{s, 0}_{p,q} (\rn)}  \asymp
\left\{\sum\limits_{j  = j_0}^{N}
2^{j (s+\frac n2-\frac np)q} \, |\lambda_{j}|^q \right\}^{1/q} \asymp \left\{\sum\limits_{j  = j_0}^{N}
2^{j (-\frac n2)q} \, |\lambda_{j}|^q \right\}^{1/q} \, .
\]
As in Lemma \ref{scalar1}, we obtain, for any $N\in\{j_0,j_0+1,\ldots\}$,
\[
\langle h_N, \cx\rangle  \asymp
 \sum_{j=j_0}^N  \lambda_{j} \,  2^{  j(- \frac n2)}\, .
\]
Then, letting $\{\mu_j\}_{j=j_0}^\fz$ be the same as in Substep 2.1 and  $\lambda_j :=
2^{\frac n2j}\mu_{j}$ for any $j\in\{j_0,j_0+1,\ldots\}$,
we can argue as before to obtain
$$\frac{|\langle h_N,\cx\rangle|}{\|h_N\|_{B^{n(1/p-1), 0}_{p,q} (\rn)}} \to \fz,\qquad N\to\fz.$$
 Thus, in this case, $\{h_N\}_{N=j_0}^\fz$ is the desired sequence.

Altogether we complete the proofs of Theorems \ref{limit5} and \ref{limit5b}.


\subsection{Proof of Corollary \ref{klassisch}\label{s5.4}}


If $p\in[1,\infty]$, then Corollary \ref{klassisch} follows directly from Theorem \ref{limit5}.
In case $p\in(0,1)$,  the if-assertion is contained in  Theorem \ref{limit5b}.
For the only if-assertion, we refer the reader to Substep 2.5) in the above proofs of
Theorems \ref{limit5} and \ref{limit5b}.


\subsection{Proof of Theorem \ref{general}\label{s5.5}}


Sufficiency can be proved as in case of $\cx$.
Necessity follows from a modification of our test functions.
As in Step 4) of the proof of Theorem \ref{charact} in Section \ref{s4},
it will be enough to deal with the existence of $\langle f, h_{1,j,0}\rangle $.
Recall,  $h_{1,j,0}=1$ on $Q_{j+1,0}$.
Now, to define the modified test functions we only select those wavelets $ \psi_{1,\ell,k}$ such that
their supports lie in between the hyperplane $x_1 =0$ and  $x_1 = 2^{-j-1}$ and have a nontrivial intersection with
the set $\{x \in \rr^n:~ x_2<0, \ldots, x_n <0\} $.
This leads to the following family
\begin{equation}
 \label{test10}
\widetilde{f}_N(x) :=   \sum_{\ell=\max\{j_0,j\}}^N  \lambda_{\ell} \hspace{-0.1cm}\sum_{\gfz{k=(0,k_2, \ldots , k_n)\in\zz^n}
{0 \le k_i < 2^{\ell-1} - \widetilde{L}, ~ i\in\{2, \ldots, n\}
}} \, \psi_{1,\ell,k} (x), \  \ \forall\,x \in \rn, \   \forall\, N \in \nn\cap [j_0,\fz).
\end{equation}
In a similar way our second family $\{{g}_N\}_{N=j_1}^\fz$ has to be modified.
Now we can proceed as in proofs of Theorems \ref{limit5} and \ref{limit5b}.
We omit further details.

\subsection*{Acknowledgements}
The authors would like to thank Professor Peter Oswald who made the manuscript \cite{Os18} available for us.
They would also like to thank the referee for her/his   careful reading and  
several valuable comments which   improved the presentation of
this article.
This project is supported by the National Natural Science Foundation of China
and the Deutsche Forschungsgemeinschaft (Grant No. 11761131002). Wen Yuan is also 
partially supported by the National Natural Science Foundation of China 
(Grant No. 11871100) and the Alexander von Humboldt Foundation. 
Dachun Yang (the corresponding author) is also 
partially supported by the National Natural Science Foundation of China (Grant Nos. 11571039 
and 11671185). 


\end{document}